\documentclass[a4paper,10pt]{article}
\usepackage[utf8]{inputenc}
\usepackage{ulem}
\usepackage{mathrsfs}

\usepackage{amssymb}
\usepackage{amsmath}
\usepackage{amsthm}
\usepackage{graphicx}
\usepackage[usenames]{color}
\usepackage{mathptmx}
\usepackage{amsmath,amsfonts,latexsym,graphicx,amssymb,url}
\usepackage{txfonts}
\usepackage{pifont,bbding,pxfonts,manfnt}
\usepackage{hyperref}
\numberwithin{equation}{section}

\newcommand{\dist}{\text{dist}\,} 

\newcommand{\di}{{\rm div}\,}

\newcommand{\sgn}{\text{{\rm sgn}}\,}

\newcommand{\cA}{ {\cal A}}
\newcommand{\cB}{ {\cal B}}
\newcommand{\cC}{ {\cal C}}

\newcommand{\cE}{ {\cal E}}
\newcommand{\cP}{ {\cal P}}

\newcommand{\cF}{ {\cal F}}

\newcommand{\cL}{ {\cal L}}
\newcommand{\cM}{{\cal M}}

\newcommand{\cS}{ {\cal S}}
\newcommand{\cX}{ {\cal X}}

\newcommand{\sM}{\mathscr{M}}

\newcommand{\vn}[1]{\left\Vert #1 \right\Vert}
\newcommand{\IP}[2]{\langle#1,#2\rangle}

\newcommand{\R}{{\mathbb R}} 
\newcommand{\bR}{\mathbb{R}}
\newcommand{\bC}{{\mathbb C}}
\newcommand{\N}{{\mathbb N}}
\newcommand{\bN}{{\mathbb N}}
\newcommand{\bZ}{\mathbb{Z}}
\def\bT{\mathbb{T}}

\def\Om{\Omega}

\def \d {\partial}

\def\de{\delta}

\def\Ltwod{\dot L^2}
\def\pHd{\dot H^2_{per}}
\def\pHa#1{\dot H^{#1}_{per}}
\theoremstyle{plain}
\newtheorem{theorem}{Theorem}[section]
\newtheorem*{theorem*}{Theorem}
\newtheorem{corollary}[theorem]{Corollary}
\newtheorem{lemma}[theorem]{Lemma}
\newtheorem{proposition}[theorem]{Proposition}
\newtheorem{definition}[theorem]{Definition}
\newtheorem{remark}[theorem]{Remark}
\definecolor{wineRed}{rgb}{0.7,0,0.3}
\newcommand{\PR}[1]{{\color{wineRed}#1}}

\definecolor{Teal}{rgb}{0.0,0.5,0.5}
\newcommand{\GW}[1]{{#1}}
\def\cC{\mathcal{C}}
\newcommand{\SE}{\mathcal{E}}

\title{Convergence of solutions to a convective Cahn-Hilliard type equation of the sixth order in case of small deposition rates}
\author{Piotr Rybka\thanks{Institute of Applied
Mathematics and Mechanics, the University of Warsaw, ul. Banacha 2,
02-097 Warsaw, PL {\tt rybka@mimuw.edu.pl}}, Glen Wheeler\thanks{School of Mathematics and Applied Statistics, University of Wollongong, Northfields Ave, 2522 NSW, Australia
{\tt glenw@uow.edu.au}}}

\begin{document}

\maketitle

\begin{abstract}
We show stabilisation of solutions to the sixth-order convective Cahn-Hilliard equation. {The problem} has the structure of a gradient flow perturbed by a quadratic destabilising term with coefficient $\delta>0$. Through application of
an abstract result by Carvalho-Langa-Robinson
we show that for small $\delta$
the equation has the structure of gradient flow in a weak sense. 
On the way we prove a kind of Liouville theorem for eternal solutions to parabolic problems. Finally, the desired stabilisation follows from a powerful theorem due to Hale-Raugel.
\end{abstract}



\section{Introduction}
In \cite{savina} Savina {\it et al.}\ 
a surface-diffusion based process describing the formation of quantum dots (and their faceting) on a growing crystalline hypersurface was considered.
By taking a thin-film approximation, the following model was proposed:
\begin{equation}\label{rna}
 \begin{array}{ll}
  h_t = \frac{\delta}{2}| \nabla h|^2 + \Delta (\Delta^2 h - 
 \di D_F W(\nabla h))& \hbox{in }\Omega \times\bR_+,\\
h(x, 0) = h_0(x)& \hbox{for } x \in \Om.
 \end{array}
\end{equation}
Here, $\Omega = \bT^d$, $d=1,2$ is a flat $d$-dimensional torus, i.e. $\bT^d = (\bR /L \bZ)^d$, where $L>0$.  
Geometrically, this imposes a periodic structure on the hypersurface.
We assume that  $W$ is a multi-well potential with the explicit form:
\begin{equation}\label{defW}
 W(v) = \frac 14(v^2-1)^2\quad\hbox{for }d=1,\qquad W(F_1,F_2)=  \frac{\alpha}{12}(F_1^4 +
F_2^4) + \frac{\beta}{2} F_1^2 F_2^2 -
\frac{1}{2}(F_1^2 + F_2^2),\quad\hbox{for }d=2,
\end{equation}
where $\alpha, \beta >0$ are anisotropy coefficients.

The evolution equation \eqref{rna} for $\delta=0$
is the $H^{-1}$-gradient flow for an analytic functional,
$\int_\Omega (\frac 12|\Delta h|^2 + W(\nabla h))\,dx$, and as such enjoys automatically a Lyapunov functional, with a number of powerful estimates as consequences.

The presence of the term $\frac{\delta}{2}| \nabla h|^2$ is known to be responsible for the introduction of spatio-temporal chaos called the Kardar-Parisi-Zhang instability, see \cite{KPZ}.
Its presence removes the gradient flow structure of the evolution equation \eqref{rna}, and as such, is the primary source of difficulty in the analysis of solutions.

There has been a concerted effort to study the parabolic problem (\ref{rna}) in the literature.
The steady states (in the case $d=1$) were studied in \cite{korzec-evans}.
In \cite{kory12,konary12} existence and uniqueness of weak solutions and a-priori estimates for $d=1$ and $d=2$ were established.
Existence of a global attractor in $H^2$ was shown in \cite{konary15} for the evolution of the derivative $u = \nabla h$, see (\ref{eqn:hcch}) and (\ref{u2D}) below, in the case of $d=1$ and $d=2$.
The existence of attractors for a similar problem was established in \cite{wangliu}.
Furthermore, in one space dimension the existence of time-periodic solutions and stability of traveling waves was studied in \cite{chhc-periodic} and \cite{TW-chhc}, respectively. More recent studies on higher-order Cahn-Hilliard equations include \cite{chl,miranville,paza,peng}.

In this paper we complete, for the first time, global analysis of solutions to \eqref{rna} for $\delta$ larger than zero.
We begin by observing that the average of any solution is a monotone function of time:
\begin{equation}
\label{rn-sr}
\frac{d}{dt}\int_\Omega h(x,t)\,dx
= \frac\de2 \int_\Omega |\nabla h(x,t)|^2\,dx \ge 0
\,.
\end{equation}
Since equality is achieved in \eqref{rn-sr} only when $h$ is constant, but we do not expect constants to be in any sense stable solutions of \eqref{rna}, this strongly suggests that the average of a solution can not be controlled.
Indeed, we expect that solutions converge to states with $|\nabla h|$ almost everywhere constant,
which in light of \eqref{rn-sr} indicates that asymptotically the average of
solutions should grow linearly in time.

This is why we factor out the influence of the average in our study of the asymptotic behavior of solutions.
To achieve this we consider not \eqref{rna} as it stands, but the spatial derivative of this system.
For $d=1$, solutions to the differentiated system enjoy an additional conservation law $\int_\Omega u\,dx =0$ resulting from the periodicity
of $h$. 
If we apply the gradient operator to both sides of \eqref{rna} and substitute $u = \nabla h$, then we arrive at
\begin{equation} \label{eqn:hcch}
\begin{array}{ll}
 u_t = \frac{\delta}{2}\nabla (u^2) + \Delta^2(\Delta u - W'(u)), & \hbox{in }
\bT^1\times \bR_+,\\
u(x,0) = u_0(x)& \hbox{for }x\in \bT^1,
\end{array}
\end{equation}
for $d=1$ while in the case where $d=2$ we obtain a system of equations:
\begin{equation} \label{u2D}
\begin{array}{ll}
 u_t = \frac{\delta}{2} \nabla |u|^2 + \Delta^3 u - \nabla\Delta\di D_u W(u_1,
u_2) & \hbox{in } \bT^2\times \bR_+\\
u(x,0) = u_0(x)& \hbox{for }x\in \bT^2.
\end{array}
\end{equation}
In \cite{konary15} the existence of a compact in $H^2$-attractor, $\cA_\de$, was
shown for both \eqref{eqn:hcch} and \eqref{u2D}.

The long time dynamics of $u$ appears to be quite sensitive to the magnitude of $\delta$.
Numerical simulations reported in \cite{konary15} 
suggest that if $\delta$ is small, then all solutions of \eqref{eqn:hcch} converge to a steady state as $t\rightarrow \infty$.
On the other hand, if $\delta$ is large, then computations show chaotic behavior of trajectories.

Our main result 
partially confirms that this is correct: we establish the existence of a $\GW{\delta_2}>0$ such that solutions to \eqref{eqn:hcch} for $\delta<\GW{\delta_2}$ converge to steady states as $t\rightarrow\infty$.
Remarkably, we are able to establish this for domains of any length outside a negligible set.
We are also able to deal with $L^2$ data, rougher than the natural energy space for the flow.
The formal statement is as follows:

\begin{theorem}\label{th-main}
There is a set $E\subset \mathbb{R}_+$ with zero Lebesgue measure and the following property. If $L\in \bR_+\setminus E$, then there exists a
positive ${\delta_2}$ 
such that for all $\delta\in(-\delta_2,\delta_2)$ and
any $u_0\in \Ltwod$ the unique solution to \eqref{eqn:hcch} with initial
data $u_0$ 
converges in $\pHa{k}$, for any $k\in \bN$, to a steady state $\bar u$ as $t\to \infty$. 
\end{theorem}
Here 
\[
\Ltwod = \bigg\{u\in L^2(0,L):\ \int_0^L u(s)\,ds =0\bigg\}\,,
\]
and
\[
\dot H^k_{per} = \bigg\{ u \in H^k(0,L)\cap \Ltwod\ :\ u(0) = u(L)\bigg\}\,,\quad \text{where $k\in\N$}\,.
\]

\bigskip

The $\delta<0$ case is stabilising and not physically relevant. However, the method we use yields the result for negative $\delta$ as well as positive $\delta$ with no additional effort, and so we have included both cases in our statement of Theorem \ref{th-main}.

For $\delta=0$, the parabolic PDE \eqref{eqn:hcch} is a gradient flow and convergence as above is expected (we also prove this in the present paper).
Clearly, for $\delta>0$ this property is lost.
However it remains intuitively tempting to consider for small $\delta$ the PDE \eqref{eqn:hcch} as being `almost' a gradient flow.
This intuition is behind our proof of Theorem \ref{th-main}.

Hale-Raugel's \cite{hale-raugel} powerful convergence result is also remarkably general, applying not only to gradient flows, but also to {\it gradient-like systems}.
To our great advantage here, Hale-Raugel \cite{hale-raugel} use a very weak notion of ``gradient-like system''.
Namely, they essentially require only that the $\omega$-limit set consists entirely of equilibria.
To show this, we 
apply a theorem of Carvalho-Langa-Robinson, cf.  \cite[Theorem 5.26]{robinson}, that is itself a structural-stability type result.
Hale-Raugel's theorem also requires that the kernel of the linearised operator at a steady state is one-dimensional. This rules out immediately the two-dimensional case, because then the kernel of the linearisation is at least two-dimensional. Moreover, when $d=1$ the
periods of the form $L=2\pi k$, $k\in \bN$ {are excluded,} because for such periods the kernel of the linearised operator is two-dimensional. Apart from this we do not know the structure of $E$, which is a result of an application of Sard's Theorem.

Now, let us give some more details on our method.
As mentioned above, for $\de=0$ the evolution equations \eqref{rna} and \eqref{eqn:hcch} do have a gradient flow structure, which we shall now explain.
Let us consider $\cF:V\rightarrow[0,\infty)$ an analytic functional over $V = \pHa3$, defined via
\begin{equation}
\label{eqncF}
\cF[u] = \frac12 \int_\bT |\nabla u|^2\,dx + \int_\bT W(u)\,dx\,.
\end{equation}
Let $H = (\pHa{2}(\bT))^*$ be equipped with the inner product
\[
(u,v)_{H} = ((-\Delta)^{-2}u,v)_{L^2(\bT)
}\,.
\]
Here we have used $(-\Delta)^{-1}$ to denote the inverse of the negative
Laplace operator, after fixing the mean to be zero, see \eqref{eqInvLap}. 
We note that $D\cF:H^3(\bT)\rightarrow H^{-3}(\bT)$.

The system \eqref{rna} for $\delta = 0$, that is
\begin{equation}\label{rnade0}
 \begin{array}{ll}
  u_t = \Delta^2 (\Delta u - W'(u))& \hbox{in }\bT
  \times\bR_+,\\
u(x, 0) = u_0(x)& \hbox{for } x \in \bT
\end{array}
\end{equation}
is equivalent to the $H$-gradient flow for the functional $\cF$. 
That is, a solution $z:[0,\infty)\rightarrow V$ to
\begin{equation}
\dot{z}(t) + \cF'(z(t)) = 0
\label{gradflow}
\end{equation}
generates a (weak) solution to \eqref{rnade0} by setting $u(x,t) = z(t)(x)$.
We note that the above setup implies that the mean of $u$ is preserved, and
that we could deduce 
(and we will do so in Corollary \ref{co-4.6}) existence of the
limit of $u$ as $t\to\infty$.

In Theorem \ref{t-count} and in  Corollary \ref{c-count} we prove, via a careful study of the problem and Sard's theorem, that there exists $L_0>0$ 
such that for $L\in[L_0,\infty)\setminus E$, 
where the exceptional set $E$ is as postulated in Theorem \ref{th-main},
we have finitely many 
non-trivial steady states of \eqref{rnade0} for $\de =0$.
Furthermore, we show that there exists at least one non-trivial steady state.
For $L\in(0,L_0)$ the only  steady state is the trivial one $v(x) = 0$.

While the trivial steady state is a singular isolated point in state space, the other equilibria are more interesting.
For a non-trivial steady state $v$, the shifted function $\tau_sv(x) = v(x+s)$ is in general not the same as $v$, yet remains a steady sate.
By periodicity in the domain, the action of $\tau$ on a non-trivial steady state generates a loop of steady states in the state space.
We define the family  $\cS$ 
to be:
\[
\cS =\{ \cE^i:\ i=0, 1,\ldots, n\}\,,
\qquad\text{where}\qquad
\cE^i = \{ \tau_s v^i\}_{s\in[0,L)},\qquad i=1,\ldots, n\,,
\]
and $v^i$,  $i=1,\ldots, n$ are the steady states assembled in Corollary \ref{c-count}.
In addition, we use $\cE^0 =\{0\}$ to represent the trivial steady state.

We show in Theorem \ref{t-1dke} that the kernel of the linearisation of
the right-hand-side of \eqref{rnade0} at a non-trivial steady state is one-dimensional. The only exception is the trivial steady state $v=0$.
There the kernel of the linearisation may be bigger or smaller depending
on $L$.
More precisely, if $L$ is not a multiple of $2\pi$, then there are positive and negative eigenvalues; if $L< L_0$ then there are only negative eigenvalues (and we have stability).
If $L$ is a multiple of $2\pi$ then the eigenspace is two-dimensional and  
our approach does not apply.
Furthermore, it is worth mentioning that if the dimension of the domain $d\ge2$, the eigenspace is again at least two-dimensional. These facts are the main reasons for our focus here on the one-dimensional flow \eqref{rnade0}.

For the singularly perturbed problem with $\delta>0$ we deduce (via Theorem \ref{t-sm-dep}) a similar picture: that for small $\de>0$ the steady state solutions to \eqref{eqn:hcch} form the same number of families:
$$
\cS^\de =\{ \cE^{\de,i}:\ i=0, 1,\ldots, n\},
$$
where
$$
\cE^{\de,i} = \{ \tau_s v^{\de,i}\}_{s\in[0,L)},\qquad i=1,\ldots, n,\qquad \cE^{\de,0}=\{0\},
$$
and $v^{\de,i}$,  $i=1,\ldots, n$ are the steady states. 
Moreover, we have that $\| v^{\de, i}- v^{0,i}\|_{H^6} \le C \de$. 
In proving this, we face a major difficulty here related to the shift invariance of the families $\cE^{\delta, j}$. This makes the linearisation of the  RHS of (\ref{rnade0}) singular. Thus, a direct use of 
the implicit function theorem is  impossible. We overcome this issue by freezing  the value of the steady states at one point.
This approach is made more complicated by the fact that the linearised operator is not self-adjoint. However, it turns out that it is 
spectrally equivalent to a self-adjoint operator; this last fact is crucial for our method.

Next, we may deduce by using the classical perturbation theory of Kato \cite{kato}
that for small $\de$ the kernel of the linearisation of the RHS of \eqref{eqn:hcch}
is one-dimensional for $v\ne0$, when $L\in [L_0,\infty)\setminus E$ or trivial in the case of $v=0$. 
This fact enables us to prove that zero is the only steady state of (\ref{eqn:hcch}) which is close to the zero equilibrium of (\ref{rnade0}) (c.f. Lemma \ref{co-dim0}).

As alluded to above, in order to deduce that equation \eqref{eqn:hcch} for $\de>0$, which is a perturbation of a gradient system \eqref{rnade0}, is a {\it gradient-like system} we use \cite[Theorem 5.26]{robinson} of Carvalho-Langa-Robinson.
For this purpose we introduce the machinery of analytic semigroup theory \cite{henry}.
We devote Section \ref{s-bare} to defining the fractional powers $H^\alpha$ of $L^2(\bT)$, $\alpha\in [0,1]$.
This task is particularly easy for flat tori. 
We also check there that the nonlinearity conforms to the requirements of the theory in \cite{henry} (see Lemma \ref{lem33}).

Since the original statement of \cite[Theorem 5.26]{robinson} is extensive, we devote most of Section \ref{s-co} to checking if \cite[Theorem 5.26]{robinson} is indeed applicable.  This task requires the existence of compact attractors $\cA_\de$ for $\de>0$ and a range of other properties of the semigroups for $\de>0$.
Once this task is complete, we may give the proof of Theorem \ref{th-main}, which is the content of Section \ref{S52}.

The main technical difficulty we must resolve here may be termed a Liouville-type or rigidity theorem for eternal solutions. It is the content of Theorem \ref{TMgap}. Roughly speaking, it says that there is a non-existence gap, that is, no eternal solution may be too close to any connected component of $\cS$. The proof involves an analysis of the time nearby trajectories take to move within a neighbourhood of an equilibrium, and with this, shows a definite reduction in the velocity of the eternal trajectory. This observation and the method of proof may be of independent interest.

\paragraph{A note on notation.} We study a one dimensional  spatially periodic problem. The partial derivatives of $u$ with respect to the variable $x$ from $\bT$ are denoted by  $\nabla u$, $u_x$, $u_{xx}$, $u_{xxx}$. However, in order to avoid overly complicated notation we write $\Delta^2 u$ for the fourth derivative  and  $\Delta^3 u$ for the sixth derivative. We additionally use the notation $\Delta_D$ for the operator $\Delta_D: D(\Delta_D) \to \dot L^2(\bT),$ where $D(\Delta_D)= H^2_{per}\cap H^1(0,L)$ and $\Delta_D u = u_{xx}$.

\section{Steady state solutions to \texorpdfstring{\eqref{eqn:hcch}}{(1.4)} for \texorpdfstring{$\de\ge0$}{delta greater than or equal to zero}}
\label{s-sss}

In this section our objective is two-fold: (1) to count the families of
equilibria and (2) to study the spectral properties of the linearisation about
each member of these families.
We first achieve this goal for $\de=0$, then we establish similar results for
small $\de>0$.
The first of the constraints on $\de$ comes from the implicit function theorem that 
we use to show that families of equilibria may be continuously
parametrised by $\de$.
In order to achieve this we need to account for the invariance of solutions
with respect to translation in the domain.
This is done by 
fixing one point (see Theorem \ref{t-sm-dep}).

\subsection{Analysis of equilibria for \texorpdfstring{$\delta=0$}{delta equal to zero}}\label{s-ss0}

We defined in \eqref{eqncF} the functional $\cF$. 
It is straightforward to check that $\cF$ is lower semicontinuous, coercive and
bounded from below.
Then, by the direct method of the calculus of variations, we deduce the
existence of at least one minimiser.
These are weak solutions to the Euler-Lagrange equation computed with respect
to the $H^{-2}$ topology, that is equation \eqref{rd0ss}.
We will revisit this later, when we discuss the gradient flow structure of \eqref{rnade0}.
Here, we will use ODE techniques, because they are more precise than the direct method of the calculus of variations, to study \eqref{rd0ss}.
In particular, we are interested in counting families of non-trivial solutions to \eqref{rd0ss}.

We proceed by reducing our sixth order equation to a more familiar second order
problem. Once we succeed we will be able to draw the phase portrait, inferring
bounds on solutions. We shall count the number of one-parameter families of solutions
and prove that zero is a simple eigenvalue for the linearised operator.

If we set $\de=0$, then the steady state equation for \eqref{eqn:hcch} in the
one-dimensional case becomes
\begin{equation}\label{rd0ss}
	\Delta^2( \Delta u - (W'(u))) =0\,,\qquad x\in[0,L),
\end{equation}
where $u$ is periodic and has zero average.

Eq. (\ref{rd0ss}) implies that there exist $a,b,c,c_1\in\R$ such that
$$
\Delta u - (W'(u)) = ax^3+bx^2+cx+c_1
\,.
$$
Since the solution is smooth, so is the LHS of the above equation.
This implies $a=b=c=0$, as the only smooth periodic polynomials are constant.
We summarise this in the following corollary.
\begin{corollary}
Any solution $u$ to \eqref{rd0ss} satisfies the following equation:
\begin{equation}\label{rd3}
 \Delta u = W'(u) + c_1\,,\qquad x\in[0,L)\,,
\end{equation}
for some $c_1\in\R$.  $\qed$
\end{corollary}
Multiplication of \eqref{rd3} by $u_x$ and integration yields
\begin{equation}\label{rd4}
 \frac{1}{2} u_x^2 = W(u) + c_1 u + c_2\,,\qquad x\in[0,L)\,.
\end{equation}
We see that $c_2$ is a constant of motion. Any solution to \eqref{rd3} is
contained in a level set of
\begin{equation}
\label{Fdefn}
	F(u,p) = -\frac 12 p^2 + W(u) + c_1 u
	       = -\frac 12 p^2 + \frac14(u^2-1)^2 + c_1 u
\,.
\end{equation}
Since we require that our solutions are periodic, we examine only those level
sets of $F(u,p)$ which are closed curves. We must determine the set $\cC$ of parameters $c=(c_1,c_2)$ for which there exists a periodic solution to (\ref{rd4}). We make first a couple of simple observations.
\begin{figure}[t]
\begin{centering}
\includegraphics[width=8cm]{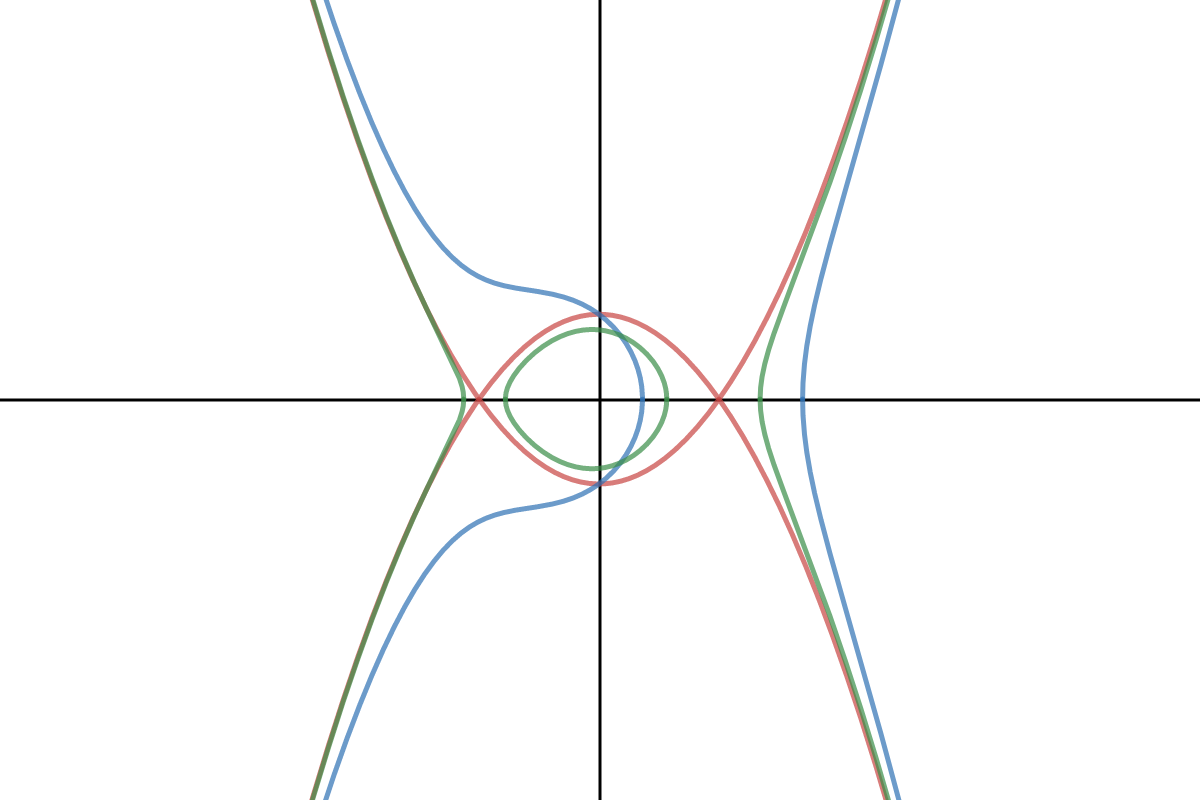}
\caption{In this figure several level curves for $F(u,p)$ (see \eqref{Fdefn}) are depicted: The parabolae for
$c_1=c_2=0$ ({\color{red}red}); the non-admissible open curves for $c_1=-1$,
$c_2=0$ ({\color{blue}blue}); and the level curves (two non-admissible
components, one admissible) for $c_1 = -\frac19$, $c_2 = -\frac16$
({\color{green}green}).}
\label{figlevelcurves}
\end{centering}
\end{figure}

We are interested in nontrivial, i.e. non-constant,  solutions $u(\cdot)$ to (\ref{rd4}), with zero mean over $[0,L)$. Thus, $u_m := \min u < 0  <\max u=: u_M$. Here is the first observation.

\begin{proposition}\label{pr-sym}
Let us suppose that $u$ is a solution to \eqref{rd4}, considered on $\Big[-\frac L2, \frac L2\Big]$ by periodicity and such that $u(0)= u_i$, where $i=m$ or $i=M$.
Then $u(-x) = u(x)$. 
\end{proposition}
\begin{proof}
We will consider only the case $i=m$, the other one is the same. Since $u$ changes sign,  
there are
$x^-<0< x^+$ such that $u(x^-)= u(x^+)=0$.

In order to shorten the notation, for a parameter $c=(c_1,c_2)$  we shall write
\begin{equation}
\label{P_cdef}
P_c(u) =  2(W(u) + c_1 u + c_2)\,.
\end{equation}
Then, for $x\in ( x^-, x^+)$ equation \eqref{rd4} takes the form 
\begin{equation}\label{rn-rev}
\frac {du}{dx} = \sgn x \sqrt{P_c(u(x))},\quad u(0) = u_m\,.
\end{equation}
Let us consider
$u(-x)$. If we take its derivative, then we have,
$$
\frac {d}{dx}u(-x) = - \frac {du}{dy}(y) = - \sgn x \sqrt{P_c(u(y))} = \sgn(-x) \sqrt{P_c(u(-x))}\,.
$$
Since $u$ is a  unique solution to \eqref{rd3} with initial conditions $u(0)=
u_m$, $u'(0)=0$, and $u(-x)$ solves the same equation, we have $u(x) = u(-x)$ as desired. 
\end{proof}

\begin{remark}\label{rmk1}
Proposition \ref{pr-sym} implies that if we know the values of $u$ on $(x_m,
x_M)$ where $u(x_m)= u_m$, $u(x_M) = u_M$, then we can extend $u$ by reflection
across the endpoints of $(x_m, x_M)$.
\end{remark}

Our goal is to provide an estimate on the number of solutions, more precisely, since each solution to \eqref{rd0ss} remains a solutions after composition with the shift operator $\tau_s$, we bound the number of {\it one parameter families} of solutions. We shall see that they are defined solely in terms of $c_1$ and $c_2$.

\begin{proposition}\label{p-shift}
Let us suppose that $u_1$ and $u_2$ are two solutions to \eqref{rd4} with zero average
and periodic with principal period $L$, corresponding to the
same values of $c_1$ and $c_2$.
Then, there exists an $s\in[0,L)$ such that $u_1(x) = \tau_su_2(x)$ for all $x\in[0,L)$.
\end{proposition}
\begin{proof}
When values of $c_1$ and $c_2$ for both solutions are the same, \eqref{rd4} implies
\[
\max \big\{x\in [0,L):\ u_1(x)\big\} = \max \big\{x\in [0,L):\ u_2(x)\big\}\,, \qquad
\min \big\{x\in [0,L):\ u_1(x)\big\} = \min \big\{x\in [0,L):\ u_2(x)\big\}\,.
\]
We can find points, where $u_1(x_j) = u_2(y_j)$, $j=1,\ldots n$, where $n\ge
2$. The values of derivatives $u_1'(x_j)$, $u_2'(y_j)$ are specified by the phase portrait of \eqref{rd4}. Due to periodicity, derivatives of $u_1$ and $u_2$ take on both positive and
negative values, thus we can find a point where $u'_1(x_k) = u'_2(y_k)$.

We take $s = y_k-x_k$. Since $u_1(\cdot)$ and $u_2(\cdot+s) = \tau_su_2(\cdot)$ satisfy the
same equation with the same initial conditions, by uniqueness for ODEs they must coincide.
\end{proof}

We wish to deduce some information about the behavior of $P_c$ on the orbits. We begin with a simple observation, whose easy proof based on the phase portrait is left to the reader.

\begin{lemma}\label{leposi}
If parameters $c=(c_1, c_2)$ are such that $P_c(u)\ge 0$ for all $u\in \bR$, then eq. (\ref{rd4}) has no periodic solutions. $\qed$
\end{lemma}

Of course,  the derivative of $u$ vanishes at the extremal points, thus $P_c(u_m) = 0 = P_c(u_M).$ At the same time Lemma \ref{leposi} asserts that $P_c$ must attain negative values for $c$ corresponding to periodic solutions. However, if $u(\cdot)$ is a periodic solution to (\ref{rd4}), then $P_c(u(x))$ must be non-negative for all $x\in [0,L)$. In this way we deduce that eq. (\ref{rd4}) has a periodic solution, provided that the parameter $c$ is such that the polynomial $P_c$ has exactly four distinct real roots,
\begin{equation}\label{roots}
u_1 < u_m<0< u_M < u_2.
\end{equation} 
Moreover, $P_c(u) > 0$ for $u\in (u_1, u_m) \cup (u_M, u_2).$

We now prove that natural restrictions on the range of $c_1$ and $c_2$ arise.
\begin{lemma}\label{cy-admiss} There exists a monotone decreasing function 
$
r: [-\frac14, 0]\to [0,\frac 13 \sqrt{\frac23}]
$
such that if $u$ is a nontrivial periodic solution to (\ref{rd4}) 
with zero mean, then 
$(c_1,c_2)\in \cC=\{(c_1,c_2):\ |c_1| < r(c_2), c_2\in (-\frac14, 0)\}$.
In other words, for each $c$ from 
$\cC$, the polynomial $P_c$ 
has four distinct zeros.  If $c\in \d \cC$, then the  polynomial $P_c$ 
has a double root.

\end{lemma}
\begin{proof}
Let us consider first $c_1>0$ and take $u$ to be a nontrivial solution, then 
$u_M=\max_{x\in[0,L)}u(x) = u(x_M)>0$. Since $u_x(x_M) =0$, we deduce $c_2<0$ from \eqref{rd4}.
A similar argument applies for $c_1<0$. 
 
Now, if $c_2<-\frac 14$, then $W(u) + c_2<0$ on $(-1,1)$ and for no
$c_1\in \bR$ is there an interval $(a,b)$ containing zero such that $P_c(u)$ is
positive on it. Therefore we have $c_2\in [-\frac14,0)$.
 
In order
for solutions of \eqref{rd4} to form a closed curve (recall that $u$ has zero
average) $P_c(u)$ must have one positive and one negative root. 
We notice that for all $c_2\in (-\frac14, 0)$ the polynomial $W(u) + c_2 =P_{(0,c_2)}(u)$ has exactly four distinct real roots, see (\ref{roots}).
The roots (\ref{roots}) are continuous functions of the coefficients of polynomial $P_c(u)$. 
Moreover, we claim that they 
are  monotone functions of $c_2$ as long as $u_1< u_m <u_M<u_2$ and $c_1$ is fixed. We restrict our attention to the case $u_m \le 0\le u_M$ and $c_1>0$. Let us us pick $c_2$ and roots $u_i(c_2)$, $i=m,N,1,2$ of 
\begin{equation}\label{2gw}
    W(u_i) + c_1 u_i+ c_2 = 0 .
\end{equation}
We take $\bar c_2>c_2$, but close to $c_2$. Of course,
$$
W(u_i(c_2)) + c_1 u_i(c_2) + \bar c_2 =\bar c_2 - c_2>0,\qquad i=m,M,1,2.
$$
In addition,
$$
\frac{dW}{du}(u_i(c_2)) + c_1 >0,\quad\hbox{for } i=m,2
\qquad
\frac{dW}{du}(u_i(c_2)) + c_1 <0,\quad\hbox{for } i=M,1.
$$
This implies that
$$
W(u_i(c_2)) + c_1 u_i(c_2) + c_2 <0 \quad\hbox{ for }\quad\left\{
\begin{array}{l}
  u\in (u_i(c_2) - \epsilon, u_i(c_2)),\quad i=m,2; \\
  u\in (u_i(c_2) , u_i(c_2) + \epsilon),\quad i= M,1
\end{array}\right.
$$
for a small $\varepsilon>0.$ Hence, for $\bar c_2>c_2$ close to $c_2$ we have
$$
u_i(\bar c_2) \in (u_i(c_2) - \varepsilon, u_i(c_2)),\ i=m,2, \qquad
u_i(\bar c_2) \in (u_i(c_2), u_i(c_2) + \varepsilon),\ i=M,1.
$$
Hence, $u_m,$ $u_2$ are decreasing and $u_M$, $u_1$ are increasing functions of $c_2$.

For a given $c_2\in (-\frac14, 0)$, there exists $r(c_2)>0$ such that for $|c_1| > r(c_2)$ equation \eqref{2gw} has only two real roots. We define $r(c_2)$ by
$$
r(c_2) = \sup\{c_1>0: u_M(c_1,c_2) < u_2(c_1, c_2)\}\,. 
$$
Continuity of the roots imply that $u_M(r(c_2), c_2) = u_2(r(c_2), c_2)$.

Thus, for $|c_1| < r(c_2)$ the equation (\ref{2gw}) has exactly four distinct real roots $u_1 < u_m<0< u_M < u_2$ and for $|c_1| = r(c_2)$ equation (\ref{2gw}) has four real roots and one of them is a double root, i.e. $u_1 = u_m$ or $u_M = u_2$.

We notice that the function $r(\cdot)$ is decreasing. Indeed, let us take $\bar c_2>c_2$. Since $u_M(r(c_2), c_2) = u_2(r(c_2), c_2)$, then we see that for all positive $u$ we have
$$
W(u) + r(c_2) u + \bar c_2 \ge W(u_2(r(c_2), c_2)) + r(c_2)u_2(r(c_2), c_2) + c_2 +
(\bar c_2 - c_2) = \bar c_2 - c_2>0.
$$
Hence, $r(\bar c_2)$ must be smaller than $r(c_2)$.
In particular, this implies that $r(\cdot)$ is invertible.

We may compute values of $r(c_2)$ for $c_2=0$ and $c_2 = -\frac14$. In this way we will obtain an estimate on the set $\cC$. It is clear that since $W$ has two double roots at $u = -1$ and $u=1$, for any $c_1>0$ equation (\ref{2gw}) will have only negative roots, so $r(0) = 0$. 

Now, we compute $r(-\frac14)$. We look for a double root of 
$P_{(c_1,-1/4)}(u)$. After straightforward computations, we see that the conditions 
$P_c(u_0) = 0 = P_c'(u_0)$ yield $u_0 = \sqrt{\frac23}$ and $c=r(-\frac14) =\frac13\sqrt{\frac23}$.
\end{proof}

The existence of at least one closed orbit is clear for $c_1=0$. It will be formally proved in Theorem \ref{t-count} below.
However, the existence of a closed level curve for $F$ defined in (\ref{Fdefn}) does not immediately imply
existence of solutions to \eqref{rd4} with zero average and period $L$.

We note that for any $x_0\in [0,L)$ and $u(x_0)= u_0$ such that $P_c(u_0)\neq
0$ 
there is a unique
solution to \eqref{rd4} in a neighborhood of $x_0$ and satisfying
$u(x_0)= u_0$. 
As the equation is invariant under the action of the shift operator, we may
without loss of generality assume that $x_0=0$.
Let us suppose now that we are looking for solutions with
principal period $L>0$. 
Such a $u$ attains each of its values (except for maximum and minimum) exactly
twice. In other words if we set $u_0 = u_m$, then
$u(\frac12 L) = u_M$. Since we can use reflection to construct the solution on $(\frac12 L, L]$, we conclude that  $u(L) = u_m$.

As a result equation \eqref{rn-rev} implies that
\begin{equation}\label{rdL}
 \int_{u_m}^{u_M}\frac{du}{\sqrt{P_c(u)}} = \frac L2\,.
\end{equation}
Using equation  \eqref{rd4} we see that the zero average condition implies
\begin{equation}\label{rdm}
 \int_{u_m}^{u_M}\frac{u\,du}{\sqrt{P_c(u)}} = 0\,.
\end{equation}
The results we presented above may be summarized in the following statement, explaining the 
restrictions imposed on solutions to \eqref{rd4} by equations (\ref{rn-rev}) and (\ref{rdm}).
\begin{proposition}\label{c-perd}
A solution to (\ref{rd4}), $u$, corresponding to a parameter $c=(c_1,c_2)\in \cC$ 
is periodic with principal period $L(c)$ and zero mean if and only if (\ref{rdm}) holds. Moreover, the period $L(c)$ is given by (\ref{rdL}). 
\end{proposition}

We next show that periodic solutions to (\ref{rd4}) exist, subject to a condition on their period.
\begin{theorem}\label{t-count}
There exist a positive $L_0$ and $E\subset[L_0, \infty)$ a set of Lebesgue measure zero such that:\\
(a) if $0<L<L_0$, then there is no solution to 
\eqref{rd4} satisfying (\ref{rdm}) with principal period $L$;\\
(b) if $L\in [L_0,\infty)\setminus E$, then the set
$$
\cP_{L}=\{(c_1,c_2)\in\cC
\,:\,\text{there is  a solution to 
\eqref{rd4} with principal period $L$, satisfying (\ref{rdm})
}\}
$$ 
is non-empty and finite. 
For $c\in \cP_L$ the principal period $L$ satisfies  (\ref{rdL}). 
\end{theorem}

Before we prove this result we make some remarks. First, we note that for any $L$ the zero function is a (trivial) solution to (\ref{rd0ss}) hence \eqref{rd4}.
But $L$ is not the (principal) period of any constant function. 
Second, this theorem implies the existence of multiple families of solutions if $L = kL_0$, where $k\in\{2,3,\ldots\}$.

\begin{corollary}\label{c-count} Let us suppose that $L\in [ L_0,\infty)\setminus E$, where $L_0$ and $E$ are defined in the previous theorem.
Set $\mathscr{S}  = \{L/k\,:\, k\in\bN\,,\ L/k\text{ satisfies Theorem \ref{t-count} and}\ L/k\ge L_0\}$.
Then, the number of one-dimensional families of solutions to \eqref{rd4} with zero mean is bounded by the finite cardinality of 
$$
\bigcup_{\{L/k\ \in\ \mathscr S\}} \cP_{L/k}
\,.
$$
\end{corollary}
\begin{proof}
If for $c\in \cC$ a function $u$ is a periodic solution to \eqref{rd4} with zero mean and period $L$, then its principal period is of the form $L/k$ for a natural number $k$. Hence, $c \in \cP_{L/k}$. By Theorem \ref{t-count} the cardinality of the set $\cP_{L/k}$
is finite and 
the number $L/k$ must be at least $L_0>0$.
\end{proof}

The statement of Theorem \ref{t-count} implies that if $L\in [L_0,\infty)\setminus E$, then all solutions to (\ref{rd3}) satisfying (\ref{rdm}) with period $L$ form a family $\cS$, where
$$
\cS=\{ \cE^{0,i}: i=0,\ldots, n \}
$$
where 
$n= n({L}),$ 
$\cE^{0,0}=\{0\}$ and
$$
\cE^{0,i} = \{\tau_x u^i\}_{s\in [0,L)},\qquad i=1,\ldots, n,
$$
where $u^i \equiv u^{c_i}$, and $c_i \in \bigcup_{\{L/k\ \in\ \mathscr S\}} \cP_{L/k}$.

We prove our theorem up to a negligible set $E$ of periods. On the one hand this is a deficiency of our method arising from the application of Sard's Theorem. On the other hand we explicitly exclude from our considerations the integer multiples of $2\pi$, because of the poor behavior of the linearised operator. As a result, we conclude that the negligible set $E$ is not empty, but at the moment we do not know any good estimate of its size.

The proof of Theorem \ref{t-count} is based on a series of lemmata. The first one establishes the character of dependence of the period on parameters $c=(c_1,c_2)\in \cC$.
\begin{lemma}\label{analiticity}
The functions in formulas (\ref{rdL}) and (\ref{rdm}) are analytic in  $\cC$, i.e. $g_1$ and $g_2$ defined below, are analytic. Moreover, $L_0 = \min_{c\in \cC} g_1(c)>0.$
\end{lemma}
\begin{proof}
For all points $(c_1, c_2)\in \cC$
(recall Lemma \ref{cy-admiss})
we set
$$
g_1(c_1,c_2):= \int_{u_m}^{u_M}\frac{du}{\sqrt{P_c(u)}} \quad\hbox{ and }
\quad
g_2(c_1,c_2):= \int_{u_m}^{u_M}\frac{udu}{\sqrt{P_c(u)}}.
$$
We will show that both functions $g_1$ and $g_2$ are not only continuous, but also analytic.

We are now going to establish continuity of $g_1$.
In order to do this 
we will simplify the integral defining $g_1$ by a change of
variable. For this purpose we factor $P_c$. Namely,
\[
P_c(u) = \frac{\sqrt 2}4(u_M - u)(u-u_m)(u_2 - u)(u-u_1),
\]
where 
\begin{equation}\label{rsta}
u_1 < u_m<0< u_M < u_2
\end{equation}
and the products of the first two factors and the last two are positive for $u\in (u_m, u_M)$.

We apply one of the Euler substitutions:
let us set
$$
y = \sqrt{\frac{u-u_m}{u_M -u}}.
$$
Then, $dy = \frac1{2y} \frac{u_M-u_m}{(u_M-u)^2} \,du$ and
$$
\int_{u_m}^{u_M}\frac{du}{\sqrt{P_c(u)}} = 
4 \sqrt 2 \int_0^\infty \frac{dy}{(1+y^2)\sqrt{Q(u(y)))}} =: g_{1}(c),
$$
and
$$
\int_{u_m}^{u_M}\frac{u\,du}{\sqrt{P_c(u)}} = 
4 \sqrt 2 \int_0^\infty \frac{u(y)\,dy}{(1+y^2)\sqrt{Q(u(y)))}} =: g_{2}(c),
$$
where $Q(u) = (u_2 - u)(u-u_1)$ and
\begin{equation}\label{rnu}
u (y) = \frac{u_m + u_My^2}{1+y^2}\,.
\end{equation}
Thus, in order to obtain a lower bound on $g$ it is sufficient to find an upper bound on $Q$. It is easy to see from the form of $Q$ that 
$$
\max Q = Q\left(\frac{u_1+u_2}2\right) = \frac 14 (u^2_2 - u^2_1) \le \frac{u_2^2}{4}.
$$
We note that $u_2$ is largest if $c_1 = -\frac{2\sqrt3}{9}$ and $c_2=-\frac
14$. It is easy to see that $u_2<2$. Hence,
$$
Q(u) \le  \frac{u_2^2}{4} < 1\qquad\hbox{for all } u\in (u_m, u_M). 
$$
As a result,
$$
g_1(c) > 4 \sqrt 2 \int_0^\infty 
\frac{dy}{(1+y^2)} 
= 2\sqrt 2\pi\qquad\hbox{for all }
(c_1,c_2)\in \cC
\,.
$$
Thus, $L_0$ is positive.

Due to the definition of  $Q$, (\ref{rsta}) and (\ref{rnu}) we infer that $Q(u(y))\in [Q_0,1]$, where $Q_0>0$.
Then, 
$$
\frac{1}{(1+y^2)\sqrt{Q(u(y))}} \le \frac{1}{(1+y^2)\sqrt{Q_0}}
$$
and the RHS is integrable.

The roots of the polynomial $P_c$ depend continuously on $(c_1,c_2)$, so if we fix $y$, then
the following function
$$
(c_1,c_2)=c \mapsto Q(y) \equiv (u_2(c) - u(y))(u(y) - u_1(c))
$$
is continuous. As a result, we deduce by the Lebesgue Dominated Convergence Theorem that
$g_1$ is continuous.

Essentially the same argument shows continuity of $g_2$. The only difference is that we notice that the factor $u(y)$ in the definition of $g_2$ is bounded for $y\in (0,\infty)$.

Since the roots $u_1$, $u_2$, $u_m$, $u_M$ of the polynomial $P_c$ are not only continuous but also analytic functions of $c$, we conclude that $Q$ and $u$ are analytic.
Therefore $g_1$ and $g_2$ are analytic as well.
\end{proof}

Knowing the boundary behavior of $g_1$ and $g_2$ is of great help. Here is our observation.
\begin{lemma}\label{infty}
If $c_0\in \d\cC$ and $c_2>- \frac14$, then
\begin{equation}\label{limg2}
\lim_{(c_1,c_2)\to  c_0}|g_i(c_1, c_2)| = \infty,\qquad i=1,2.
\end{equation}
\end{lemma}
\begin{proof} 
Keeping in mind the definition of $Q$ we notice that if
$c\to c_0$, where $c_0\in \partial\cC$, $c_0=(c_{01}, c_{02})$ and $c_{02}>-\frac14$, then either
$$
\lim_{c\to c_0} u_m(c) = \bar u_m = \lim_{c\to c_0} u_1(c),\quad \lim_{c\to c_0} u_M(c) = \bar u_M <
\bar u_2 =\lim_{c\to c_0} u_2(c)
$$
or
$$
\lim_{c\to c_0} u_m(c) = \bar u_m > \bar u_1= \lim_{c\to c_0} u_1(c),\quad \lim_{c\to c_0} u_M(c) = 
\bar u_M =\lim_{c\to c_0} u_2(c).
$$
The first case corresponds to negative $c_1$ while the second one occurs when $c_1>0$. Since the problem is symmetric, it is sufficient to consider one of them, e.g. the first one. In this case
$$
\lim_{c\to c_0} Q(u) =(\bar u_M-\bar u_m) \frac{y^2}{1+y^2} 
\left(\bar u_2 - \frac{\bar u_m + \bar u_M y^2}{1+y^2}\right).
$$
Thus, the Fatou Lemma implies
$$
\lim_{c\to c_0} g_1(c) \ge 
\int_0^\infty \lim_{c\to c_0}\frac{4\sqrt2\, dy}{(1+y^2)\sqrt{Q(u(y))}}\ge
\int_0^\infty \frac{4\sqrt2\, dy}{\sqrt{(1+y^2)}y\sqrt{u_M- u_m}\sqrt{\bar u_2 - \bar u_m}}= \infty. 
$$

Now, we take care of $g_2$.
We notice that function $y\mapsto \bar u_m + y^2 \bar u_M$ has a zero at $y_0 = \sqrt{-\bar u_m/\bar u_M}.$ We can write $g_2(c)$ as a sum the integral over $(0,\frac{y_0}2)$ and $[\frac{y_0}2,\infty)$. We notice that for $c$ sufficiently close to $c_0$ the integrand is negative over $(0,\frac{y_0}2)$ and integrable over $[\frac{y_0}2,\infty)$. We notice that again with the help of the Fatou Lemma we deduce that
 $$
 \lim_{c\to c_0} \int_0^{y_0/2} \frac{|u_m|}{(1+y^2)Q(u(y))}\,dy \ge 
 \int_0^{y_0/2}\frac{|\bar u_m|\,dy}{(1+y^2)^{3/2}y\sqrt{\bar u_M-\bar u_m}\sqrt{\bar u_2 - \bar u_m}} = \infty.
 $$
 As a result we find 
$$
\lim_{c\to c_0}| g_2(c)| = \infty.
$$
\end{proof}

Now, we will {\it prove Theorem \ref{t-count}}. 
As a restatement of the definition of $\cP_{L}$, we know that $c\in \cP_L$ if and only if  the following pair of equations 
\begin{equation}\label{syseq}
 g_1(c_1,c_2) = \frac{L}{2},\qquad g_2(c_1,c_2) = 0   
\end{equation}
are satisfied. We must show that there is only a finite number of solutions in $\cC.$

First of all we show that 
if $L\ge L_0$, then there is $\tilde c\in \cC$ such that (\ref{rdL}) and (\ref{rdm}) hold.
Indeed, if we take $c_1=0$, we deduce from the phase portrait of (\ref{rd4}) that $u_m = - u_M$. Moreover, by Proposition \ref{pr-sym} identity (\ref{rdm}) holds, i.e. $\{(0,c_2): c_2\in [-\frac 14, 0]\} \subset g_2^{-1}(0)$. In addition, for all $c_2\in (-\frac14, 0)$ we have $g_1(0,c_2)\in[L_0, \infty)$ so there is $c' \in (-\frac14, 0)$ such that $g_1(0,c') = L/2$. As a result $\cP_L$ contains $(0,c')$.

Now, we will specify the exceptional set $E$. Let us denote by $Z$ the set of critical values of $g_1$.
Since $g_1$ is analytic in $\cC$ we may apply the Sard theorem to it. As a result we deduce that $Z$ 
has Lebesgue measure zero. We set
$$
E = Z \cup \{2k\pi\}_{k=1}^\infty.
$$
We take any $L\in [L_0,\infty) \setminus E$. Then $\sM = g_1^{-1}(\frac{L}2)$ is a compact one-dimensional manifold. We know that any compact component of a one-dimensional manifold is homeomorphic to either a line segment or a circle. Moreover, if $\sM_1$ is a component of $\sM$, which is topologically a line segment, then $\d\sM_1 \subset \d\cC$. However, due to Lemma \ref{infty} we deduce that $\d\sM_1 \subset \{c\in \d\cC: c_2 = - \frac14\}$.

Since $\sM$ is compact it has a finite number of components. Now, we want to count the number of points in the intersection 
$$
\sM \cap g_2^{-1}(0).
$$
Since $\sM$ is an analytic manifold, then $g_2$ restricted to $\sM$ is an analytic function of one variable. So it may have only a finite number of zeros or it vanishes identically on a component of $\sM.$ 

In order to rule out the last possibility, we notice that the functions $g_1$ and $g_2$ are linearly independent on $\cC$. Indeed, if $\lambda_1 g_1 + \lambda_2 g_2 \equiv 0$ in $\cC$, then we examine this identity at $(0,c_2)$. At this point $g_2$ vanishes, so we are left with $\lambda_1 g_1(0,c_2) = 0$. Since $g_1$ is always non-negative, then $\lambda_1 =0$ so our claim follows.

The linear independence of $g_1$, $g_2$ implies that level sets of these functions may not coincide. 
\qed

\bigskip
We stress that we leave out from our period considerations  all the integer multiples of $2\pi$. We notice that if the period $L$ grows so does the set of periodic solutions, however it remains finite. We stress that if $L<L_0$, then
the only steady state is the trivial one.

\bigskip
Studying the eigenvalue problem for the linearised operator is of utmost
importance for the analysis of the asymptotic behavior of solutions to the
evolution problem. 
Let us suppose that $u$ is a steady state of \eqref{eqn:hcch}, that is it satisfies
\begin{equation}\label{rdss}
0= \Delta^2(\Delta u - (W'(u))),\qquad x\in [0, L)
\end{equation}
with periodic boundary conditions and zero average. We will also write $u^*$ to make a distinction or $u^j$ to stress that this equilibrium belongs to family $\cE^{0,j}$.

The linearisation of the right-hand-side (RHS) above is
\begin{equation}\label{lop}
\cL_u \tilde w = \Delta^2(\Delta \tilde w - W''(u)\tilde w)
\end{equation}
with period boundary conditions. Occasionally, we will suppress the subscript of $\cL$, when no ambiguity arises or the specific $u$ is not important. 

Of course \eqref{rdss} and its solution $u$
are invariant under the shift operator $u \mapsto\tau_s u$.
Tangential to this action is  $\tilde w = u_{x}$, which is therefore an eigenvector of $\cL_u$ with zero eigenvalue.
We will prove that this is the only eigenvector.
We first treat the case that $u$ is not trivial.
The trivial case is dealt with in Corollary \ref{cor-2.10} below.

Before we give the theorem and its proof, let us note an important fact on the multiplicity of zeros for equilibria.

\begin{proposition}\label{pr-sgnv} 
Let us suppose that $u$ is a non-trivial steady state. The multiplicity of any zero for $u$ is one; that is, all zeros of $u$ are simple.
\end{proposition}
\begin{proof}
A double zero of $u$ occurs if the point $(0,0)$ belongs to the level set of the function $-p^2 + P_c(u)$. This happens only if $c_1= 0$ and $c_2=- \frac 14$. This set of parameters corresponds to the trivial solution $u \equiv 0$.
\end{proof}

\begin{theorem}\label{t-1dke}
Let us suppose that $W$ is given by \eqref{defW} and $u$ is a nontrivial solution to \eqref{rdss} and $\cL$ is defined in  \eqref{lop}.
Then the only periodic solution, $w$ with period $L$ to $\cL_u w= 0$
with zero mean value is a multiple of $u_x$.
\end{theorem}
\begin{proof}
By integrating and using periodicity plus zero average as before, we deduce that $w$ 
must satisfy the following equation:
\begin{equation}\label{rd5}
\Delta w - W''(u) w = a,  
\end{equation}
where $a$ is a constant of integration. Now, we multiply eq. (\ref{rd5}) by $u_x$ and we subtract from it the result of
differentiation of \eqref{rd3} and its multiplication by $w$. 
Then, we obtain
\begin{equation}\label{rd6}
	u_x \Delta w - \Delta u_{x} w = au_x\quad\text{or}\quad
(u_x w_x)_x = (w\Delta u)_x + au_x\,.
\end{equation}
Hence, we reach
\begin{equation}\label{rd7}
u_x w_x = w\Delta u + au + b, \qquad x\in [0,L), 
\end{equation}
where $b$ is another constant of integration. We treat \eqref{rd7} as a linear equation for $w$ with an initial condition at any $x_0\in[0,L)$.
We require that any solution $w$ of this equation is periodic with period $L$
and also that it has zero mean.

By the shift invariance, we may without loss of generality assume that $u(x_0)=0$ and $x_0=0$.
Then, the usual variation of parameters technique yields
\begin{equation}\label{rd8}
 w(x) = u_x(x) \int_0^x \frac{a u(s)+ b}{u_x^2(s)}\,ds + d u_x(x),
\end{equation}
where $d$ is a constant of integration. We note that due to \eqref{rd4} we have
$u_x(0)\neq 0$, so the integral is well defined at least near $x=0$. Obviously, the integrand is not integrable near any $x=\bar x$, where $u_x$ vanishes.

For the moment we take \eqref{rd8} on an interval containing zero.
In order to proceed we fix $u$ by requiring (in addition to $u(0)=0$) that $u$ attains its minimum $u_m$ at $x_m<0$ and maximum $u_M$ at $x_M>0$. Moreover, on intervals $(x_m,0)$ and $(0, x_M)$ the function $u$ is monotone.
Using the reflection technique mentioned in Remark \ref{rmk1} we extend $u$ to the real line.

Note that for $x\in[x_m,x_M]$ the product
$$
u_x(x) \int_0^x \frac{a u(s)+ b}{u_x^2(s)}\,ds =: w_0(x)
$$
is uniformly bounded.
Indeed, if $x\ne x_m$ or $x_M$ then $u_x$ is bounded away from zero. Let $x^k < x_M$ be a monotone sequence $x^k\rightarrow x_M$.
Note that simplicity of the zeros of $u$ (see Proposition \ref{pr-sgnv}) implies that $u_x(x^k)^2/(x_M - x^k)^2$ converges to a positive constant as $k\rightarrow\infty$.
Therefore, for sufficiently large $k$ and $x\in(x^k,x_M)$ we have
$$
| w_0(x)| \le |u_x(x)|\left| \int_0^{x^k} \frac{a u(s)+ b}{u_x^2(s)}\,ds \right| +
|u_x(x)|\left| \int_{x^k}^{x} \frac{C_0}{(s-x^k)^2} \,ds\right| \le C_1|u_x(x)|  + 
C_1\frac{|u_x(x)|}{|x-x^k|}\le K
\,.
$$
The above estimate shows that $w_0(x_M)$ is also bounded.
The estimate for $w_0(x_m)$ proceeds similarly.

Our claim will follow once we show that $w_0\equiv 0$. 
Let us examine $w_1$ and $w_2$ defined as
$$
w_1(x) = u_x(x) \int_0^x \frac{ u(s)}{u_x^2(s)}\,ds, \qquad
w_2(x) = u_x(x) \int_0^x \frac{1}{u_x^2(s)}\,ds.
$$
Recall that $u$ has been constructed using the reflection technique.
This implies that $w_1$ is non-negative {for all $x\in [x_m, x_M]$}; but then, it may not be a solution, because $w_1$ would then be required to have zero average.

So, we must have $a=0$. Then, we see from \eqref{rd7} that
$$
\left(\frac{w_2}{u_x}\right)_xu_x^2 = b\,.
$$
The left-hand-side (LHS) of this formula must vanish somewhere, and is periodic, so it must not have constant sign.
But the RHS does have constant sign (or vanishes if $b=0$).
Therefore neither $w_1$ nor $w_2$ may be present, and so we must have $w_0 = 0$ and $w(x) = du_x(x)$. This completes the proof.
\end{proof}

Finally, we deal with the case left out in Theorem \ref{t-1dke}.
\begin{corollary}\label{cor-2.10}
(a) If $L = 2k\pi$, then kernel of the linearisation, $\cL_{u^0}$, around a trivial solution $u^0$ is two-dimensional.\\
(b) If $L<L_0$, then the kernel of $\cL_{u^0}$ is trivial.
\end{corollary}
\begin{proof}
We note that for the trivial steady state we must analyse
$$
\Delta w + w = a\,.
$$
We see immediately that the condition $\int_0^L w\,dx =0$ implies that $a=0$. The
solution space to this equation is contained in
\[
\{\alpha\sin(x+\beta)\,:\,\alpha,\beta\in\R\}
\,.
\]
If $L=2k\pi$ for some $k\in\N$, the kernel is this entire two-dimensional space and so (a) follows.
Otherwise, i.e. if if $L\neq 2k\pi$,  the kernel is trivial, which shows  part (b).
\end{proof}

We shall now study the spectrum of the linearised operator in more detail.
We note that even for $\delta=0$ the linearised operator $\cL_{u^*}$, where $u^*$ is a steady state, is not self-adjoint.
However, inspired by the work of Bates and Fife \cite{bates-fife,bates}, we expect
that the eigenvalue problem for $\cL_{u^*}$ is equivalent to an eigenvalue problem
for a self-adjoint operator. 
Before doing so, we 
recall the form of the inverse Laplacian $\Delta^{-1}: \Ltwod  \to \pHd$ as
\begin{equation}
\label{eqInvLap}
\Delta^{-1} f(x) = \int_0^L K(x,y) f(y)\,dy\,,
\end{equation}
where
$$
K(x,y) = \Big(x-y\Big)_+ +\Big(y-L\Big)\frac xL + \frac y2\Big(\frac{y}{L} -1\Big)
$$
and $a_+$ denotes $\max\{a, 0\}$.

The explicit form of the kernel $K$ permits us to make the following observation. 
\begin{corollary}
If $f \in \pHa k$, 
where $k\ge 0$, then $\Delta^{-1} f \in  \pHa {k+2}$. 
Here, $\pHa0$ denotes $\Ltwod$.
\end{corollary}

We additionally remark that a similar formula for the operator $\Delta_D: D(\Delta_D) \to \dot L^2(\bT),$ where $D(\Delta_D)= H^2_{per}\cap H^1(0,T)$ and $\Delta_D u = u_{xx}$ also holds. We leave the detail to the interested reader.

We are going to discuss  the spectral properties of $\cL$, the operator defined in (\ref{lop}),
at a steady state $u^*$.
We will consider $\cL$ with a natural  domain $D(\cL) =\pHa{6}$.
We denote by $\cL_0$ the operator $\cL$ restricted to a smaller domain, namely
\begin{equation}\label{de-lo}
\cL_0 h = \cL h\qquad\hbox{for}\quad h\in D(\cL_0)= \pHa6 \cap H^1_0(0, L).
\end{equation}

In order to state our results  we introduce another pair of operators, $\cM$, $\cM_0$ with domains
$$
D(\cM) = \pHa{6}, \qquad \hbox{and}\qquad D(\cM_0) = \{ h\in\pHa{6}: \Delta h(0) = 0 = \Delta h(L) \}\,,
$$
and given by 
$$
\cM_0 h = \Delta( \Delta_D\Delta h - \Delta_D(q \Delta h))
\qquad\hbox{and}\qquad
\cM h = \Delta( \Delta^2 h - \Delta(q \Delta h)).
$$

We now establish the following result.
\begin{proposition}\label{p2.18}
(a)
 Let us suppose that $\lambda\in \bC$. Then, $f\in \Ltwod$ and $w\in \pHa6$
 satisfy
 $$
 \cL_0 w - \lambda w = f
 $$
 if and only if
\[
\cM_0 \Delta^{-1} w - \lambda \Delta^{-1} w = \Delta^{-1} f
\]
holds.\\
(b) The operator $\cM_0$ is self-adjoint, so $\sigma (\cL_0) =\sigma (\cM_0)\subset\bR$ and the spectrum is discrete; moreover, if $L$ does not belong to the exceptional set defined in Theorem \ref{t-count}, then the zero eigenvalue of $\cL_0$ is simple.\\
(c) Statements (a) and (b) above are valid for $\cM$ in place of $\cM_0$.
\end{proposition}
\begin{proof}
In the notation we use here, we have 
$\cL_0 w = \Delta \cM_0 \Delta^{-1}w$.
Therefore
\[
f = \cL_0 w - \lambda w
= \Delta \cM_0 \Delta^{-1}w - \lambda w
\,,
\]
and so
\[
\cM_0 \Delta^{-1}w - \lambda \Delta^{-1}w = \Delta^{-1}f\,.
\]
Part (a) of our claim follows. 

We notice that 
\begin{equation}\label{r-equiv}
\cL_0 u = \Delta \cM \Delta^{-1}u 
\end{equation}
for all $u\in D(\cL_0)$.
Importantly, note that for $u\in D(\cL_0)$ we have $\Delta^{-1}u \in D(\cM_0)$. 

Now, we will show that $\cM_0$ is self-adjoint on $\Ltwod$.
The domain $D(\cM_0)$ is dense in $\Ltwod$. Indeed, let $u_n\to v$ in $\Ltwod$, where $u_n$ are smooth functions. We may modify this sequence, $\bar u_n$, so that $\bar u_n\in D(\cM_0)$. Namely, we set 
$$
v_n''(x) = \eta_n(x) u_n'',
$$
where $\eta_n$ are smooth cut-off functions converging a.e. to $1$. We set
$$
w_n (x) = \int_0^x\int_0^s v_n''(t)\,dtds
$$
and 
$$
\bar u_n (x) = w_n(x) - \frac1L \int_0^L w_n(s)\, ds \in D(\cM_0).
$$
Thus, we conclude that $\bar u_n \to v$ in the $L^2$-norm.

It is straightforward to see that $\cM_0$ is symmetric, because we can easily perform integration by parts (the boundary terms drop  out due to the periodicity).

We also notice that $\cM_0$ is closed. Let $u_n\in D(\cM_0)$ converge to $u$ and $Mu_n \to f$ in $L^2$. Then, we may apply $\Delta^{-1}: \dot L^2 \to \dot H^2_{per}$ to the sequence $\cM_0 u_n$ to obtain
$$
\Delta_D\Delta u_n - W''(u^*)\Delta u_n 
\to \Delta^{-1} f.
$$
Thus, the limit $u$ satisfies
$$
\Delta_D \Delta u = \Delta^{-1} f + W''(u^*)\Delta u .
$$

The RHS belongs to $H^{-2}$, because $W''(u^*)\Delta u$ is a well-defined functional over $H^2$ due to the boundedness and regularity of $u^*$. 
Since $\Delta_D\Delta u \in H^{-2}$ we deduce that $u\in \dot H^2_{per}$. On the way we use the unique solvability of $\Delta v = g$  with Dirichlet boundary conditions. Such a solution need not have zero average, but we solve one more eq. $\Delta w = v$ with the restriction of the zero mean. This is done uniquely.
The conclusion that $u\in \dot H^2_{per}$ permits us to apply the bootstrap argument yielding that $u\in D(\cM_0)$. Thus, $\cM_0$ is closed.

In order to establish that $\cM_0$ is self-adjoint it suffices to check that the deficiency indices of $\cM_0$ are equal to zero, see \cite[Thm. 7.1.10]{BsimonVol4}. This means that we have to show that
$$
\hbox{dim}\, (\hbox{Rg}(\cM_0\pm i I))^\perp =0.
$$
Suppose that $f\in (\hbox{Rg}(\cM_0\pm i I))^\perp)$. Since $f\in \Ltwod$ due to density of $D(\cM_0)$  there is a sequence $v_n\in D(\cM_0)$ converging to $f$ in $H$. Then,
$$
0= \langle M v_n \pm i v_n , v_n \rangle =
\int \Delta^3 v_n \bar v_n - \Delta ( W''(u^*) \Delta v_n) \bar v_n \pm i |v_n|^2.
$$
The real and imaginary parts must be equal to zero, in particular $\int | v_n|^2\,dx = 0$. Hence $f =0$ and we conclude that $\cM_0$ is self-adjoint.

Now, we see that $\Ltwod = \hbox{Rg}\, \cM_0 \oplus \ker \cM_0 = \hbox{Rg}\, \cM_0,$ because $\ker \cM_0=\{0\}.$
Since $\hbox{Rg} \cM_0 = \Ltwod$, the equivalence of $\cL_0$ and $\cM_0$ implies that $\hbox{Rg}\, \cL_0=\Ltwod$, as desired. 

Since $\cL_0$ is defined everywhere on $\pHa6\cap H^1_0(0,L)$ with trivial kernel and onto $\Ltwod$ we deduce that the inverse of $\cL_0$ is a continuous operator.

The fact that zero is a simple eigenvalue of $\cL_{0}$ follows from Theorem \ref{t-1dke}.

Part (c)  follows immediately from (a) and (b).
\end{proof}

We showed in Corollary \ref{c-count} that there is a finite number of families of solutions to \eqref{rd0ss}.
We can identify a specific family with any of its members, since 
the orbit of the shift operator $\{\tau_s u\}_{s\in[0,L)}$ is the entire family.
Therefore we can choose $n$ canonical solutions to \eqref{rd0ss} 
 $\{u^{0,i}\}_{i=1}^n$ such that the entire solution space is generated by the action of the shift operator applied to
 $\{u^{0,i}\}_{i=1}^n$.

\subsection{Analysis of equilibria for positive \texorpdfstring{$\delta$}{delta}}\label{s-ss1}

We wish to prove that solutions of
\begin{equation}\label{rnst}
\frac{\delta}{2} (u^2)_x + \Delta^2(\Delta u - W'(u)) = 0
\end{equation}
form smooth curves parametrised by $\delta$.
For this purpose we wish to use the Implicit Function Theorem applied to 
the LHS above. However, we know that $\cL_{u^*}$, i.e. the linearisation of the LHS of (\ref{rnst}) at a steady state $u^*$ for $\delta=0$ has a nontrivial kernel. In order to resolve this issue we will restrict the domain of the linearisation. For this purpose we introduce
$$
\cX = \{ u\in \pHa6:\ u(0) = u(L) = 0\}
$$
and we define
the mapping 
$$
G:(-1,1) \times \cX \to \Ltwod
$$
by the following formula
\begin{equation}\label{defG}
G(\delta, u ) = \delta u u_x + \Delta^2 (\Delta_D u-  W'(u))\,,
\end{equation}
where as noted earlier $\Delta_D$ denotes the Laplace operator with the domain
$D(\Delta_D) = H^2_{per} \cap H^1_0(0,L)$.
We have to study the differentiability properties of $G$ on the space $\cX$ endowed with the norm of $\pHa6.$ Here are our observations.
\begin{lemma}\label{lem221}
Let us suppose that $F:\bR \to \bR$ is smooth and $g\in \pHa{k}(\bT),$ $k\ge1$, then $F(g)\in H^1(\bT)$.
\end{lemma}
\begin{proof}
The argument is based on the observation that 
$H^1(\bT)$ is embedded into $L^2(\bT)$. The details are left to the reader.
\end{proof}
\begin{lemma}\label{lem222}
If $f,g \in \pHa{k}$, $k\ge1$, then for $0\le i<k$ we have
\begin{equation}\label{eq-pro}
\| (fg)^{(i)}_x \| \le C \|f\|_{\pHa{k}} \|g\|_{\pHa{k}}.
\end{equation}
\end{lemma}
\noindent{\it Proof.}
We proceed by induction with respect to $i.$ If $i=0$ our claim follows from the Poincar\'e inequality.

Let us suppose now that (\ref{eq-pro}) holds for $i\ge 0.$ Then, we notice that
$$
\|(fg)^{(i+1)}_x\| \le \|(f_x g)^{(i)}_x\| + \|(f g_x)^{(i)}_x\|.
$$
Now, we apply the inductive assumption to the RHS above for $i<k$. Thus,
$$
\|(fg)^{(i+1)}_x\| \le C \| f_x\|_{H^{k-1}}\| g\|_{H^{k-1}}
+ C \| f\|_{H^{k-1}}\| g_x\|_{H^{k-1}} \le C \| f\|_{H^{k}}\| g\|_{H^{k}}. \eqno\qed
$$

We need one more observation.

\begin{lemma}\label{lem223}
Let us suppose that $F:\pHa{k}\to L^2$ is given by
$$
F(u) = (u^l)^{(i)}_x,
$$
where $i<k$. Then,
$$
DF(u)h = l(u^{l-1}h)^{(i)}_x.
$$
\end{lemma}
\begin{proof}
We notice that
$$
(u+h)^l - u^l = h \sum_{j=0}^{l-1}(u+h)^j u^{l-1-j}
= l h u^{l-1} + h^2 \sum_{j=2}^l \binom{l}{j} h^{j-2}u^{l-j}.
$$
Now, Lemmata \ref{lem221} and \ref{lem222} imply that 
$$
DF(u)h = l( u^{l-1}h)^{(i)}_x.
$$
Moreover, this is a continuous linear operator.
\end{proof}

\bigskip
We want to show that after fixing $ u^*\in \cE^{0,i}$ eq. (\ref{rnst}) for small positive $\delta$ has a solution not belonging to $ \cE^{0,i}$. 
Our subsequent analysis is founded on our results for the case of $\delta=0$ obtained in \S 2.1. 

\begin{lemma}\label{k-lem}
Let us suppose $u^*$ is one of the nontrivial solutions constructed in Theorem \ref{t-count}. If $G$ is defined in (\ref{defG}), then:\\
(a) $G$ is a $C^1$ mapping;\\
(b) $\ker D_2 G(0, u) = \{0\}$;\\
(c) $D_2 G(0, u):\cX \to \Ltwod$ 
is a linear isomorphism.
\end{lemma} 

\begin{proof}
By Lemma \ref{lem223}
$D_2G(\delta, u): \cX  \to \Ltwod$ has the following form:
\begin{equation*}
    D_2 G(\delta, u)h = \delta (u h)_x + \Delta^3 h 
+ \Delta^2(W''(u))h).
\end{equation*}
We recall that $G$ is linear with respect to $\delta$ hence $D_1G$ exists. Moreover, 
Lemma \ref{lem223} yields continuity of $D_2G$.

In order to prove part (b) we recall that by Theorem \ref{t-1dke} any element of the kernel, $w$, must be of the form $w = c u_x$. Since all zeros of the non-trivial steady states of (\ref{eqn:hcch}) are simple, see 
Proposition \ref{pr-sgnv}, we infer that $w$ does not belong to $\cX$ unless $w=0$.

For a proof of part (c) we {notice that $D_2G(0,u) = \cL_0$, where $\cL_0$ is defined in (\ref{de-lo}), and we have to}
address the question of solvability of 
$\cL_0 h =f$  
in $\cX$ for any $f\in \Ltwod$.
{For this purpose we recall Proposition \ref{p2.18}, which guarantees that  $\cL_0$ has  a  trivial kernel in $\cX$ and its image is $\Ltwod$. Thus, we deduce that the inverse of $\cL_0$ is a continuous operator.}
\end{proof}

After these preparations we may show that
for all small $\de>0$ there are $n$ families of steady states 
$\{u^{\delta,i}\}_{i=1}^n$ of \eqref{eqn:hcch}, analogous to the situation for $\de=0$. The theorem below introduces a constraint on $\delta$ coming from an application of the Implicit Function Theorem.
This is notable as it is one of the very few times that we are required to restricted the magnitude of $\delta$.
\begin{theorem}\label{t-sm-dep}
Let us suppose that $u^i$ is representative of a non-trivial  steady state constructed in Theorem \ref{t-count} belonging to $\cX$, when 
$L \in [L_0,\infty) \setminus E$.
Then, there exists a $\delta_0>0$ such that for $|\de|<\delta_0$ there is a $C^1$ map
$\delta\mapsto u^{\de,i} \in \cX$ 
such that $G(\de, u^{\de,i}) =0$.
Moreover, we have the estimate 
 \begin{equation}\label{co-1}
  \| u^{\delta,i} - u^i \|_{\cX} \le C\delta.
 \end{equation}
\end{theorem}
\begin{proof}
Applying the implicit function theorem from Nirenberg \cite[Theorem 2.7.2]{NirenbergBook} (note that we checked in Lemma \ref{k-lem} that $D_2G$ satisfies the assumptions of \cite[Theorem 2.7.2]{NirenbergBook}), we find a $C^1$ function $(-\delta_0, \delta_0) \ni \delta \mapsto u^{\delta,i}  \in \cX$ such that $H(\delta, u^{\delta,i}) =0$. In particular, (\ref{co-1}) follows.
\end{proof}
{We may now introduce families of the steady states of (\ref{rnst}),
$$
\cE^{\delta,0}=\{0\},
\quad \cE^{\delta,1}=\{ \tau_s u^{\delta,1}\}_{s\in[0,L)},\ \ldots, \ 
\cE^{\delta,n}=\{ \tau_s u^{\delta,n}\}_{s\in[0,L)},
$$
corresponding to $\cE^{0,0}, \ldots, \cE^{0,n}$ introduced earlier.}

We record below a fact that will be needed in Section 5. Namely, the estimate (\ref{co-1}) in Theorem \ref{t-sm-dep} implies that the equilibrium sets $\cE^{\delta,i}$ converge to $\cE^{0,i}$ as $\delta\rightarrow0$, that is:
\begin{corollary}
\label{rmkHD}
Let $\cE^{0,j}$ be an equilibrium set for \eqref{eqn:hcch} with $\delta=0$. The symbol $\cE^{\delta,j}$ for $\delta>0$ denotes the family of steady states constructed in Theorem \ref{t-sm-dep}.
Then,
\[
\lim_{\delta\rightarrow0} \dist_H(\cE^{\delta,j},\cE^{0,j}) = 0\,,
\]
where $\dist_H$  denotes the Hausdorff distance {in the metric of the ambient space $\pHa6$}.
\end{corollary}

\bigskip
{If $u^*$ is a solution to (\ref{rnst}), then the linearisation of the LHS of this eq.\! is denoted by $\cL^\delta_{u^*}$ and the superscript $\delta$ will be never suppressed to distinguish between $\cL^\delta_{u^*}$ and $\cL_{u^*} = \cL^0_{u^*}$. Let us also note that
$\cL^\delta_{u^*}h = \delta (u^* h)_x + \cL_{u^*}h$ for $h\in \pHa6$. When we want to stress that $u^*$ is in the family $\cE^{\delta, i}$, then we write $\cL_i^\delta$.

By definition of the families $\cE^{\delta,i}$, any two of its members $u^{\de,i}_1$ and $u^{\de_i}_2$,  are related by a shift, i.e.,} there exists an $s\in[0,L)$ such that
\[
	u^{\de,i}_2(x) = \tau_{s}u^{\de,i}_1(x)
	\,.
\]
Therefore, 
for the operators $\cL_{u^{\de,i}_1}^\delta$, $\cL_{u^{\de,i}_2}^\delta$, we have
\[
	\cL_{u^{\de,i}_2}^\delta w =
	\cL_{\tau_su^{\de,i}_1}^\delta w  =
	\left(\tau_s\cL_{u^{\de,i}_1}^\delta\right) w\,.
\]
In this way, we may think of the family of linearisations around a given family generated by the action of the shift operator on a given member $u^{\de,i}$ as being itself also generated by the action of the shift operator on the linearisation $\cL_{u^{\de,i}_1}^\delta$ of the representative.

{In our further analysis we will need information about the kernel of} 
$\cL_i^\delta$. 
We notice that the
projections onto the eigenspaces of $\cL^\de_i$ depend continuously on $\de$.
In particular, \cite[Ch. IV, Theorem 3.16]{kato} implies the following fact, {which notably brings a second restriction on the size of $\delta$}.
\begin{lemma}\label{co-dim0}
There exists a $\delta_1{\in(0, \delta_0]}$ such that for $|\delta| < \delta_1$ and $L\in [L_0, \infty)\setminus E$, 
the zero-eigenspace for $\cL^\de_i$
corresponding to the steady states $\cE^{\delta,i}$, for all $i=1,\ldots, n$, has dimension one. The zero-eigenspace for$\cL^\de_0$ has dimension zero.
\end{lemma}
\begin{proof}
First, let us assume that $\delta_1 \le \delta_0$ where $\delta_0$ is as in Theorem \ref{t-sm-dep}.
Then, using Theorem \ref{t-sm-dep}, the map $\delta\mapsto u^{\de,i}\in \cX$ is continuous for each $i$. We analyse only the case $i>0$.
We aim now to prove that the difference $\cL^\delta_i - \cL^0_i = A$ is a
relatively bounded bounded operator (with respect to $\cL^0_i$).
That is:
\begin{equation}\label{co-reb}
\|A w \| \le a \|w\| + b\|\cL^0_i w\|\,, 
\end{equation}
where $\|\cdot \|$ is the $L^2$ norm. We find
\[
	A w = \Delta^2((W''(u^{i,0}) - W''(u^{i,\delta}))w) + \delta (u^{i,\delta} w_x + w u^{i,\delta})
	    = 3\Delta^2((u^{i,0})^2 - (u^{i,\delta})^2)w)+ \delta (u^{i,\delta} w_x + w u^{i,\delta})
	\,.
\]
Of course, we can find $a$ and $b$ depending on $u^{i,0}$, $u^{i,\delta}$, $\delta$ and $L$ such that
$$
\| A w\| \le a'  \|w\| + b'\| \Delta^2w\| 
\le a \|w\| + b\|\cL^0_i w\|\,.
$$
Moreover, due to \eqref{co-1} we may can deduce that the dependence of $a$ and $b$ on $\delta$ is linear; that is,
\begin{equation}\label{co-2}
 \| A w\| \le C\delta ( \|w\| + \| \cL^0_i w\|)\,,
\end{equation}
where now $C$ depends only on
$u^{i,0}$, $u^{i,\delta}$, and $L$.
We wish to apply \cite[Ch. IV, Theorem 3.16]{kato} to conclude that there exists a $\tilde\delta>0$ such that for $|\delta|<\tilde\delta$ the kernel of $\cL^\delta_i$ is one-dimensional for $u^{i,\delta} \ne0$ and trivial  otherwise.

Due to \cite[Ch. IV, Theorem 2.23 (c)]{kato} it is sufficient to show that, if we fix $\lambda_0$ in the resolvent set $\rho(\cL_i^0)$ of $\cL_i^0$, then 
we have $\lambda_0\in \rho(\cL^\delta_i)$ and
\begin{equation}\label{co-re}
 \| (\cL^\de_i - \lambda_0)^{-1} - (\cL^0_i - \lambda_0)^{-1}\| \to 0\quad\hbox{ as }\quad \delta \to 0.
\end{equation}
Before establishing \eqref{co-re} we will show that $\lambda_0\in \rho(\cL^\delta_i)$.
We note that
$$
\cL^\de_i - \lambda_0 = \cL^0_i - \lambda_0 + (\cL^\de_i - \cL^0_i) = 
(Id +  (\cL^\de_i - \cL^0_i)  (\cL^0_i - \lambda_0)^{-1}) (\cL^0_i - \lambda_0).
$$
Now, \eqref{co-2} implies that $\|(\cL^\de_i - \cL^0_i)  (\cL^0_i - \lambda_0)^{-1}\| <1$ for small $\delta$, so $\lambda_0\in \rho(\cL^\delta_i)$ follows.

In order to show \eqref{co-re} we write,
\begin{align*}
(\cL^\de_i -\lambda_0)^{-1} - (\cL^0_i -\lambda_0)^{-1} &=
( \cL^0_i  -\lambda_0 + \cL^\de_i- \cL^0_i )^{-1} - (\cL^0_i -\lambda_0)^{-1}
\\&=
(Id  + (\cL^\de_i- \cL^0_i) (\cL^0_i -\lambda_0)^{-1} - Id) (\cL^0_i -\lambda_0)^{-1}
\,.
\end{align*}
Thus, our claim follows again from \eqref{co-2}.
Taking $\delta_1 = \min\{\delta_0,\tilde\delta\}$ finishes the proof.
\end{proof}
With the help of this Lemma we can show that $u=0$ is an isolated steady state of (\ref{eqn:hcch}) for $\delta>0$.
\begin{proposition}
Suppose $|\delta|<\delta_1$ and $L\in [0,\infty)\setminus E$.
Then $u=0$ is an isolated solution of
\begin{equation}\label{eq:delta}
0=\delta u u_x + \Delta^3 u- \Delta^2 W'(u)
\end{equation}
in 
$\pHa6$.
\end{proposition}
\begin{proof}
Let us suppose otherwise. Then there exists a $v\in \pHa6$ solving (\ref{eq:delta}), that is close to $u=0$. Of course, $v$ is shift invariant, i.e. $\tau_s v$ is another solution to (\ref{eq:delta}). Hence, the argument which we used earlier to analyse \ref{rdss} and \ref{lop}
implies that $v_x$ is in the kernel of $\cL_u$, i.e., the linearised operator at $v$. 

However, by Lemma \ref{co-dim0} the linearised operator has trivial kernel for $|\delta|<\delta_1$. This contradicts the observation above.
\end{proof}

\section{A reformulation of \texorpdfstring{\eqref{eqn:hcch}}{(1.4)}}
\label{s-bare}

We would like to recast our problem \eqref{eqn:hcch} in the framework of abstract
theory for analytic semigroups and dynamical systems, which facilitates
application of the theory exposed in \cite{robinson}.
The key goal of the present section is to define the function spaces required
and study the properties of relevant maps on these spaces.

Our equation \eqref{eqn:hcch} can be viewed from different perspectives. Here, we will
write it as an abstract dynamical system: 
\begin{equation}\label{rn:ds}
 \begin{array}{l}
 \dot z + A z = F_\de(z), \\
 z(0) = z_0.
 \end{array}
\end{equation}
In our presentation we will follow the theory of Henry \cite{henry}.

We consider this equation in $Z_0 = \Ltwod$ or $Z_2 = \pHa2$ depending upon our needs.
The choice of $Z_0$ seems more natural, when presenting the results of \cite{kory12} and \cite{konary12} in the unifying framework of analytic semigroup theory. On the other hand, the choice $Z_2= \pHa2$ is stipulated by the  construction of a global attractor in \cite{konary15} and the use of the theory exposed in \cite{robinson}, especially in Subsection \ref{subsec-d} of the present paper.

Since we are using an abstract framework, we can develop both cases $Z_0$ and $Z_2$ simultaneously, at least for some time.
We define $A: D(A)\subset Z_k \to Z_k$ by the following formula:
\begin{equation}\label{rn:da}
Az  = - \Delta^3 z\qquad\hbox{for }z \in D(A) = \pHa{6+k}(\bT)\subset Z_k\,,
\end{equation}
where $k=0$ or $k=2$. 

Since the domain $\bT$ is { a flat torus}, an easy argument based on Fourier series implies:
\begin{corollary}\label{lem30}
The operator $A:D(A)\subset Z_k\to Z_k$ is sectorial \footnote{See  \cite[\S 1.3]{henry} for a definition.}.
\end{corollary}
Then, following \cite{henry}, we define $A^\alpha$, the fractional powers of
$A$. Once we have them, we define $Z_k^\alpha$, the fractional powers of
$Z_k$: They are equal to $D(A^\alpha)$ equipped with the graph norm.
Actually, we can identify them with Sobolev spaces on a torus.
\begin{lemma}
If  $\alpha >0$, then $Z_k^\alpha = \pHa{6\alpha+k}(\bT)$.
\end{lemma}
\begin{proof}
The operator $A$ is self-adjoint and the Fourier transform makes it diagonal:
$$
\widehat{ (A u)}(\xi) = |\xi|^6 \hat u(x)\,.
$$
Thus, 
$$
\widehat{ (A^\alpha u)}(\xi) = |\xi|^{6\alpha} \hat u(x)\,.
$$
{Since} the fractional power $Z_k^\alpha$ is $D(A^\alpha)$ equipped with the graph norm, 
the above formula {implies} 
that $Z^\alpha_k = \pHa{6\alpha+k}$.
\end{proof}

The nonlinearity appearing in \eqref{eqn:hcch} is here called $F_\de$ (see \eqref{rn:ds}). We recall its expression here:
\begin{equation}\label{rn:nl}
 F_\de(u) =
 \delta uu_x - \Delta^2 W'(u) 
\,.
\end{equation}
We must show that $F_\de:Z_k^\alpha\to Z_k$ is well-defined. Here is our
first observation:
\begin{lemma}
\label{lem33}
If $W$ is given by (\ref{defW}) and $\alpha \in [5/6,1)$, then $F_\de$ defined in \eqref{rn:nl} is of class $C^1$.
\end{lemma}
\begin{proof} For the given range of $\alpha$ we know that $\pHa5\subset Z^\alpha_2$. Thus, it is sufficient to establish that $F_\delta: \pHa{5+k}\to Z_k$ is continuously differentiable. This follows from Lemma \ref{lem223}. 
\end{proof}
The information we gathered on $A$ as well as on the nonlinearity $F_\delta$ prompts us to state the existence of solutions to (\ref{eqn:hcch}). We have:
\begin{proposition}\label{pr34}
Suppose that $u_0 \in Z_k$. 
For any $\delta\in \bR$ there exists a unique global solution $u$ to (\ref{eqn:hcch}) such that 
$u\in C([0,\infty); Z_k^\alpha)\cap C((0,\infty); D(A))$, $u_t\in C((0,\infty);Z_k)$ and  the parameter variation formula
\begin{equation}\label{eq-convar}
u(t) = e^{\Delta^3 t}u_0 + 
\int_0^t e^{\Delta^3 (t-s)}\left[\delta u(s) u_x(s) - \Delta^2 W'(u(s))\right]\,ds
\end{equation}
holds. Moreover, $S_\delta(t)$ defined as $S_\delta u_0:= u(t)$, where $u$ is the solution, is a strongly continuous semigroup.
\end{proposition}
\begin{proof}
Since $A$ is sectorial and $F_\delta$ is (better than) locally Lipschitz continuous, we deduce by \cite[Theorem 3.3.3]{henry} local-in-time existence and the representation formula (\ref{eq-convar}). This argument does not depend upon the sign of $\delta$. Moreover, we obtain explicitly the smoothing effect, that is, if $z_0\in Z^\alpha_k$, then for any $t>0$ we have $z(t)\in D(A)$.

This statement for $u_0\in Z_0$ improves the local-in-time existence of weak solutions obtained in \cite{kory12} for data in $\pHa2$. The global-in-time existence of weak solutions was established in \cite{konary15}, again the sign of $\delta$ is unimportant. 
Strictly speaking, global existence was obtained for initial data in $Z_2$, however, it holds for any data $u_0\in Z_0$, because of the regularising effect mentioned in the previous paragraph.

Finally, the construction of \cite[Theorem 3.3.3]{henry} provides us with a strongly continuous semigroup $S_\delta(t) u_0 := u(t)$ in $Z_0$ or $Z_2$. The fact that $S_\delta$ is  strongly continuous in $\pHa2$ was showed in \cite[Proposition 3]{konary15}.
\end{proof}

\section{Structural analysis}\label{s-sa}

We wish to use \cite[Theorem 2.4]{hale-raugel} to deduce convergence of
solutions to \eqref{eqn:hcch}. For this purpose we have to check that the
hypotheses of this theorem are satisfied. The crucial one, requiring that
\eqref{eqn:hcch} is a gradient-like system, will be established in the next section by invoking a theorem of Carvalho-Langa-Robinson \cite{robinson}.
Here, we deal with all of the others. 

\begin{remark} We know that for $\delta=0$ the problem
\eqref{eqn:hcch} is already a gradient flow. As a consequence of the work in this
section we will be able to apply \cite[Theorem 2.4]{hale-raugel} to conclude the
new result that solutions to \eqref{eqn:hcch} with $\delta=0$ converge to a
steady state.
\end{remark}

To begin with, we recall the result we mentioned above.
\begin{theorem}\label{TH-R}{\rm \cite[Theorem 2.4]{hale-raugel}}
Let us the consider the problem
\begin{equation}\label{dHR}
    \dot z = C_1z + H_1(z),\qquad z(0) = z_1,
\end{equation}
where
\begin{quote}
    (H.1) $-C_1: D(C_1) \subset Z \to Z$ is sectorial and there is an $\alpha>0$ such that $H_1: Z^\alpha \to Z$ is of class $C^1$ and $D(C_1) = D(C_1 + DH(z_1))$. 
\end{quote}
We assume that  for  any $z_1 \in Z^\alpha$ its positive orbit of (\ref{dHR}) is precompact and the $\omega$-limit set $\omega(z_1)$ consists of equilibria of  (\ref{dHR}). Suppose that for any $z_0\in \omega(z_1)$, the semigroup $T(t)= T_{z_0}(t)$ generated by the linear operator $C = C_1 + DH_1(z_0)$ satisfies the hypotheses (H2), (H3) and (H4) below: 
\begin{quote}
{(H.2)} There is a decomposition $Z = X \oplus Y_1 \oplus Y_2$ with associated continuous projection operators $P_0$, $P_1$, $P_2,$ each of which commute with $T(t)$;

(H.3) The ranges $X$ and $Y_2$ of $P_0$ and $P_2$ are finite dimensional: dim\, $X= m_0$, dim \,$Y_2 = m_2$;

(H.4) The spectrum $\sigma(T(1))$ can be written as $\sigma(T(1))=\sigma_- \cup \sigma_0 \cup \sigma_+$, where  $ \sigma_- = \sigma(T(1))P_1$, $\sigma_0 = \sigma(T(1))P_0$, $\sigma_+ = \sigma(T(1))P_2$ lie inside, on, and outside the unit circle centered at $0\in\mathbb{C}$ respectively. Moreover the distance of $\sigma_-$ to this unit circle is positive.
\end{quote}
In addition, $\sigma_0 $ is either empty or it contains the only point 1, which is a simple eigenvalue of $T(1)$. Then, there is a unique point $\varphi = \varphi_{z_0}$ such that $\omega(z_1) = \{\varphi\}.$
\end{theorem}

Initially, we check that the hypotheses of Theorem \ref{TH-R} hold for \eqref{eqn:hcch} written as \eqref{dHR} in the cases $Z = Z_0 = L^2$ and
{and $Z=Z_2\equiv \pHa{2}$. 
The space $Z_0$ seem more natural, however, some results will require us to take $Z=Z_2$. 

We note that in both cases of $Z$} 
the norm is smooth, $C_1 = -A$, with $D(C_1) = \pHa6$, $H_1(z) = F_\delta(z)$, where $A$ is defined in \eqref{rn:da} and $F_\de$ is given by \eqref{rn:nl}. One reason to keep in mind $Z= Z_2 = \pHa2$ is that the results of \cite{konary15} are stated using the $Z_2$-topology.

\begin{lemma}[Hypothesis (H.1)] 
The operator $A$ is sectorial on $Z_k$,  and $F_\de$ is a $C^1$ map from  $Z_k^{{5/6}}$ 
to $Z_k$ {for $k=0, 2$}. 
\end{lemma}
\begin{proof}
We have already noted in Section \ref{s-bare} that $A$ is a sectorial operator.
In Lemma \ref{lem223} we showed that $F_\de$ is indeed a $C^1$ map defined on $Z_k^{{5/6}}$ 
with values in $Z_k,$  {for $k=0, 2$}. 
This establishes the lemma.
\end{proof}

\begin{lemma}[Hypotheses (H.2) and (H.3)]
There is a decomposition $Z_{k} =X \oplus  Y_1 \oplus Y_2$, {$k=0, 2$,} with associated continuous projection operators $P_0$, $P_1$, $P_2$  that commute with the semigroup generated by the linearisation of \eqref{rn:ds}.
The ranges $X$ and $Y_2$ of $P_0$ and $P_2$ are finite-dimensional, with dimension $m_0$ and $m_2$ respectively.
\end{lemma}
\begin{proof}
Let us use $\cL^\de$ to denote the linearisation of \eqref{rn:ds} at a steady state $u_0$.
This operator is given by 
\begin{equation}\label{def-cL}
\cL^\de= -A + DF_\de(u_0)\,.
\end{equation} 
Since $-\cL^\de$ is a compact perturbation of a sectorial
operator, it is sectorial as well.

We will call by $T(t)$ the semigroup generated by $\cL^\de$, that is, $T(t) = e^{\cL^\de t}$. We wish to show that the underlying Hilbert space $Z_{k}$
decomposes as
\begin{equation}\label{co-dec}
Z_{k} = X \oplus Y_1 \oplus Y_2, \qquad {k=0, 2,}
\end{equation}
where the corresponding continuous projection operators $P_1,$ $P_0$, $P_2$ commute with $T(1)$. 

The operator $\cL^\delta: D(A) \subset Z_{k}\to Z_{k}$ is bounded above, has a
discrete spectrum, and so only has a finite number of eigenvalues with positive
real part. Let $\Gamma_2$ denote a contour in $\bC$ contained in the upper half
plane $\{z:\Re z> 0\}$.
Let $\Gamma_0$ be a contour containing only the eigenvalue zero. Then, we
define $P_2$ (respectively $P_0$) as the integral of the
resolvent $R(\cL^\de, \xi)$ over $\Gamma_2$ (respectively $\Gamma_0$), see
\cite[Ch. 1, formula (5.22)]{kato}. We set $P_1 = Id - P_0 -P_2$. It is easy to
check that $P_1$ is a projection. Moreover, these three projections commute
with $T(1)$ because $T(t)$ is also defined through a contour integration.

Since each eigenvalue of $\cL^\de$ has a finite multiplicity and $\Gamma_2$ contains only a finite number of them we conclude that $P_2$ has finite dimensional range $Y_2$ and set $\dim Y_2=m_2$.
The range of $P_0$ is the kernel of $\cL^\de$.
We showed in Lemma \ref{co-dim0} that the kernel is one-dimensional or trivial, so $m_0 = \dim X$ is zero or one.
\end{proof}

\begin{lemma}[Hypothesis (H.4)]
The spectrum $\sigma(T(1))$ of $T(1)$ can be written as
$$
\sigma(T(1)) = \sigma(T(1)P_1) \cup \sigma(T(1)P_0) \cup \sigma(T(1)P_2) =
\sigma_- \cup\sigma_0\cup\sigma_+,
$$
where $\sigma_-$ is inside the unit circle with center at 0 in $\bC$, $\sigma_0$
is on it, and $\sigma_+$ is outside it.
Moreover, the distance of $\sigma_-$ to the unit circle is positive, and $\sigma_0 $ is either empty or it contains the only point 1, which is a simple eigenvalue of $T(1)$.
\end{lemma}
\begin{proof}
Note that, as $T(1) = e^{\cL^\delta}$, if $\lambda$ is an eigenvalue for $\cL^\delta$ then $e^\lambda$ is an eigenvalue for $T(1)$.
Now, since $Y_2$ is a direct sum of eigenspaces corresponding to eigenvalues of $\cL^\delta$ with
positive real part, $Y_2$ is invariant for $T(1)$ defined by contour
integration. Then by the functional calculus the absolute values of all eigenvalues in the spectrum of $T(1)P_2$ are greater than one. By the same token the spectrum of $T(1)P_1$ is inside the unit disc. It is  at a positive distance from the unit circle, because the only accumulation point of the eigenvalues is zero.

Since zero, if it is an eigenvalue at all, is a simple, isolated eigenvalue of
$\cL^\de$, $X$ is invariant for $T(1)$ and 1 is the eigenvalue of $T(1)$ there.
\end{proof}

{After fixing a steady state $u^\delta_0$, we have the corresponding linearisation $\cL^\delta.$}
Then, we introduce 
$$
\cC = P_0 {\cL^\delta},\qquad\cB=(Id-P_0) {\cL^\delta},
\qquad\cB_1= P_1  {\cL^\delta},\qquad\cB_2= P_2  {\cL^\delta}
$$
Hale-Raugel \cite[\S 2.1]{hale-raugel} gives us that the equation \eqref{eqn:hcch} can be written in the following coordinate system:
\begin{eqnarray}\label{co-r}
 \dot x &=& \cC x +f(x,y_1,y_2),\nonumber\\
 \dot y_1 &=& \cB_1 y_1 + g_1(x,y_1,y_2),\\
 \dot y_2 &=& \cB_2 y_2 + g_2(x,y_1,y_2),\nonumber
\end{eqnarray}
where $f= P_0(F_\delta - {F_\delta(u^\delta_0) - D F_\delta(u^\delta_0)})$, $g_1= P_1(F_\delta - {F_\delta(u^\delta_0) -D F_\delta(u^\delta_0)})$, and $g_2= P_2(F_\delta - {F_\delta(u^\delta_0) -D F_\delta(u^\delta_0)})$.
We can draw a conclusion about (\ref{eqn:hcch}) which is of independent interest, when $\delta =0$. Here it is:
\begin{corollary}\label{co-4.6}
Consider (\ref{eqn:hcch}), when $\delta=0$ and $u_0\in Z_k$, {$k =0, 2$}.
Then $\omega(u_0)=\{\phi\}$, where $\phi$ is a steady state of (\ref{eqn:hcch}).
In other words, all solutions converge to a steady state as $t\to \infty$.
\end{corollary}
\begin{proof}
Our claim \GW{follows} from Theorem \ref{TH-R}. Indeed, assumptions (H.1)--(H.4) have been already checked. Theorem \ref{TH-R} stipulates that $u_0\in Z^\alpha$, but the flow of (\ref{eqn:hcch}) is regularising, so even if $u_0\in Z,$ then for any $t>0$ we have $u(t_0)\in Z^\alpha$. Moreover, $\omega(u_0) = \omega(u(t_0)).$ 

We have to check that $\omega(u(t_0))$ consists of the steady states. However, this follows immediately from the fact that $\cF$ defined in (\ref{eqncF})
is a Lyapunov functional.
\end{proof}

\begin{remark}
We note that a result corresponding to Corollary \ref{co-4.6} can also be deduced in two dimensions. We cannot apply the Hale-Raugel theory, because the kernel of the linearisation at a steady state is at least two dimensional. We may, however, use the ideas of \L ojasiewicz, see \cite[\S 4]{13}, \cite[\S IV.9]{14}. This program was completed for the classical Cahn-Hilliard equation in \cite{hory}. Here, an adjustment of the argument in \cite{hory} is required.
\end{remark}

\section{A convergence result}\label{s-co}

Here, we prove our main result, Theorem \ref{th-main}.
This  will confirm the numerical observations that were made in \cite{konary15}.  We will
prove this in a number of steps delineated below. Our analysis is for $L\in (0,\infty)\setminus E$, where the exceptional set $E$ is as in Theorem \ref{t-count}, in particular $E$ has zero Lebesgue measure. We have no tool to address $L\in E$:
{f}or $2k\pi \in E$, $k\in \bN$, $k>0$, we may say that the argument presented here is not applicable, because the linearised operator $\cL_u$ at $u =0$ has a two-dimensional kernel.

The case $L\in(0,\GW{L}_0)$ is also special, because the zero solution is the only steady  state.

We use the stability of gradient flows under small perturbations by
\cite{robinson}. Although the notion of gradient flow in \cite{robinson} is weak, it
is strong enough to allow us to apply the Hale-Raugel convergence result.
Below we check that \cite[Theorem 5.26]{robinson} applies.
We recall here the notions used in \cite{robinson}.
The first one is that of the global compact attractor.
\begin{definition}{\rm (\cite[Definition 1.5]{robinson})}\\
A set $\cA \subset Z$ is the global attractor for a semigroup $S (\cdot):Z \to Z$ if\\
(i) $\cA$ is compact;\\
(ii) $\cA$ is invariant; \\
(iii) $\cA$ attracts each bounded subset of $Z$.
\end{definition}
We also have to explain what we mean by a gradient flow here.

\begin{definition} \label{defRob}{\cite[Definition 5.3, Definition 5.4, Theorem 5.5]{robinson}}
We say that a semigroup $S$ with a global attractor $\cA$ is a gradient flow with respect to the family $\cS=\{\cE^{0},\ldots,\cE^{k}\}$ of invariant sets provided  that:\\
1) For any global (eternal) solution $\xi:\bR \to Z$ taking values in $\cA$, there exist $i,j\in\{0,\ldots, k\}$ such that
$$
\lim_{t\to -\infty} \dist(\xi(t), \cE^{i}) =0\quad\hbox{and}\quad
\lim_{t\to \infty} \dist(\xi(t), \cE^{j}) =0.
$$
2) The collection $\cS$ contains no homoclinic structures.
\end{definition}
We will see at the end of this section that in the case of problem (\ref{eqn:hcch}) this notion 
is stronger than that used in \cite{hale-raugel}. We recall that Hale-Raugel require only that the omega-limit set consists only of steady states. Thus, provided we can prove that \eqref{eqn:hcch} is a gradient flow in the sense of Definition \ref{defRob}, the results of \cite{hale-raugel} are applicable to \eqref{eqn:hcch}.

\subsection{Stability of the gradient flow \texorpdfstring{\eqref{eqn:hcch}}{(1.4)} at \texorpdfstring{$\delta=0$}{delta equals zero}}\label{s5.1}

Let us recall:

\begin{theorem}\label{thmRob} {\rm (\cite[Theorem 5.26]{robinson}: Stability of gradient semigroups).}
Let $S_0(\cdot)$ be a semigroup on a Banach space $Z$ that has a global attractor
$\cA_0$ and that is a gradient flow with respect to the finite collection $\cS^0$ of isolated invariant sets $\{\cE^{0,0} ,\cE^{0,1} ,\ldots, \cE^{0,k}\}$. Assume that:
\begin{enumerate}
\item[(a)] for each $\delta\in(0,1]$, $S_\delta(\cdot)$ is a semigroup on $Z$ with a
global attractor $\cA_\delta$;
\item[(b)] $\{S_\delta(\cdot)\}_{\delta\in[0,1]}$ is collectively asymptotically compact and
$\overline{\bigcup_{\delta\in[0,1]}\cA_\delta}$ is bounded;
\item[(c)] $S_\delta(\cdot)$ converges to $S_0(\cdot)$, in the sense that
\[
d(S_\delta(t)u, S_0(t)u) \rightarrow 0\qquad\text{ as }\qquad\delta\rightarrow0
\]
uniformly for $u$ in compact subsets of $Z$; and
\item[(d)] for $\delta\in(0,1]$, $\cA_\delta$ contains a finite collection of isolated invariant sets $\cS^\delta = \{\cE^{\delta,0} ,\cE^{\delta,1},\ldots,\cE^{\delta,k}\}$ such that
\[
\lim_{\delta\rightarrow0} \dist_H(\cE^{\delta,j},\cE^{0,j}) = 0
\]
and there exist $\eta > 0$ and $\delta_1\in(0,1)$ such that for all $\delta\in(0,\delta_1)$, if $\xi_\delta:\R\rightarrow \cA_\delta$ is a global (or eternal) solution, then
\[
\sup_{t\in\R} \dist( \xi_\delta(t), \cE^{\delta,j})\le\eta 
\quad
\Rightarrow
\quad
\xi_\delta(t)\in \cE^{j,\delta}\text{ for all $t\in\R$.}
\]
\end{enumerate}
Then, there exists a $\delta_{2} \in (0,\delta_1)$ such that, for all
$\delta\in(0,\delta_\PR{2})$, $S_\delta(\cdot)$ is a gradient
semigroup with respect to $\cS^\delta$. In particular
\[
\cA_\delta = \bigcup_{i=1}^k W^u(\cE^{\delta,i})\,.
\]
\end{theorem}

We now verify each of the parts (a)--(d) of Theorem \ref{thmRob}. At this point, we are forced to choose $Z= Z_2$ and the family of semigroups $S_\delta (\cdot):Z_2\to Z_2$ defined in \cite{konary15}. It turns out that this topology is suitable to prove part (d).

We have to put together results from different sources, while taking care that we are consistently using the metric of the Theorem above.

\subsubsection{Part (a) of Theorem \ref{thmRob}}\label{s51}

In \cite{konary15}, the authors proved the existence of a global  compact attractor $\cA_\delta$ for each $\delta>0$ for the semigroup $S_\delta(\cdot)$ acting on $Z= Z_2 \equiv \dot{H^2}_{per}$. This is sufficient for our purposes. We remark that with a slight effort one could show that $\cA_\delta$ is also an attractor for the semigroup considered on $Z_0$. Since we do not use this result, we are omitting the details.

\subsubsection{Part (b) of Theorem \ref{thmRob}}\label{subsec-b}

Here, we have to check that the strongly continuous semigroups $S_\delta(\cdot)$ on $Z_2$  associated to (\ref{eqn:hcch}) are {\it collectively asymptotically compact}. 

\begin{proposition}\label{pr53}
The family of semigroups $S_\delta(\cdot)$ is collectively asymptotically compact, that is:
Let $\{t_n\}$ be a sequence such that $t_n\to\infty$ and let $\{x_n\}\subset Z_2$ as well as 
$S_{\de_n}(t_n)x_n$ be bounded.
Then, $S_{\de_n}(t_n)x_n$ contains a convergent subsequence.
\end{proposition}
\begin{proof}
We shall see that the assumption of the boundedness of  $S_{\de_n}(t_n)x_n$ is redundant.

In the course of the proof of \cite[Theorem 10]{konary15} it was shown that if $u_0\in Z_2$, then 
\begin{equation}\label{eq-bdat}
    \frac d{dt}E_1 + \epsilon E_1 \le C_6\,,
\end{equation}
where 
$$
E_1(u) = \int_\bT W(u)\,dx + \frac12 \| u_x\|^2 + 2C_1 \| (-\Delta)^{-1} u\|^2
$$
and $C_6 = C_6(L,E_1(0))$, $C_1 = \delta^2 C_0$, where $C_0$ is a universal constant, and $\epsilon$ is small.

The estimate (\ref{eq-bdat}) implies that for  $u_n(t)= S_{\de_n}(t)x_n$ we have
\begin{equation}\label{eq-S}
   \int_\bT W(u_n(t))\,dx + \frac12 \| (u_n)_x(t)\|^2 \le  E_1(u_n(t)) \le \frac{C_6}{\epsilon}\qquad\text{for all } t\ge 0\,.
\end{equation}
This in turn yields
\begin{equation}\label{eq-S+}
\| u_n (t)\|_{\pHa1}\le \frac{C_6}{\epsilon}\qquad\text{for all } t\ge 0\,.
\end{equation}
We stress that this estimate is uniform in $\delta\in (0,1)$ and in $n\in\bN$ due to the uniform boundedness of $x_n$ in $Z_2$.

We notice that if we set $v_n = \PR{\Delta^{-2}(}\frac\delta 2 (u^2_n)_x) - W'(u_n)$, then (\ref{eq-S+}) implies that
\begin{equation}\label{eq-S++}
 \|\nabla v_n(t)\|_{L^2} \le C_1 \qquad t\ge 0.   
\end{equation}
Now, we are going to use  repetitively the constant variation formula (\ref{eq-convar}).  
If we combine (\ref{eq-S++}) with (\ref{eq-convar}), then we see
\begin{equation}\label{eq-unib}
\| u_n(t) \|_{H^2} \le M e^{-\lambda t}\| x_n\|_{H^2}
+ \int_{0}^t M \frac{e^{-\lambda(t-s)}}{(t-s)^{\frac56}}\|v_n(s) \|_{H^1}\,ds \le C_2, \qquad t\ge 0,
\end{equation}
where $C_2$ is a universal constant. In other words, we deduced a uniform  boundedness of $S_{\delta_n}(t_n)x_n$ in $Z_2$.

We are interested in estimating the integral term in (\ref{eq-convar}) in $\pHa3$. Due to (\ref{eq-unib})
we see that
$$
\| u_n(t_n) - e^{\Delta^3 t_n}x_n\|_{H^3} \le 
\int_0^{t_n} M \frac{e^{-\lambda(t-s)}}{(t-s)^{\frac56}}\|v_n(s) \|_{H^2}\,ds \le C_3 .
$$
The estimate above shows that we can select a convergent in $Z_2$ subsequence of $u_n(t_n) - e^{\Delta^3 t_n}x_n$. In addition,  $e^{\Delta^3 t_n}x_n$ goes to zero in $Z_2$. Our claim follows.
\end{proof}

\begin{remark}\label{rem-unib}
We notice that (\ref{eq-unib}) shows that if the initial condition is in $Z_2$ and $u$ is the corresponding solution, then the norm $\| u(t)\|_{Z_2}$ is bounded in terms of $\|u_0\|_{Z_2}$.
\end{remark}

\subsubsection{Part (c) of Theorem \ref{thmRob}}\label{subsec-c}

We have to show that $S_\de(\cdot)$ converges to $S_0(\cdot)$, as $\de\to0$ in the sense of Theorem \ref{thmRob}. We recall that existence of the semigroups $S_\delta(\cdot)$, $\delta\in\R$ is stated in Proposition \ref{pr34}.
Indeed, we have 
\begin{lemma}\label{lpartc}
Let us suppose that $K\subset Z_2$ is compact.
Then for each $t \ge 0$
$$
\lim_{\delta\to 0}\dist_{Z_2}(S_\delta(t) u_0, S_0(t) u_0) =0
$$
uniformly for $u_0\in K$.
\end{lemma}
\begin{proof}
For a fixed $u_0\in K$, set $u = S_\delta(\cdot) u_0$ and $v = S_0(\cdot)u_0$.
We establish it in a few steps. 

{\it Step 1.} The following estimate is valid for any $t>0$:
\begin{equation}
\label{EQcclaim}
	||(-\Delta)^{-1}(u-v)||_{L^2}^2(t) \le C\delta 
					\,.
\end{equation}
Given \eqref{EQcclaim} the result will follow because we have already established uniform estimates on $\|\Delta u\|_{L^2}$ and $\|\Delta v\|_{L^2}$, established in the course of proof of Proposition \ref{pr53}.
Indeed, set $w = u-v$ and we notice
\begin{eqnarray*}
 \| u- v\|^2_{L^2}(t) &= &\int_0^L w w \,dx = \int_0^L w (-\Delta) (-\Delta)^{-1} w \,dx
 = \int_0^L  (-\Delta) w (-\Delta)^{-1} w \,dx \\
& \le &\|\Delta(u- v)\|_{L^2}(t) \|(-\Delta)^{-1}(u- v)\|_{L^2}(t) .
\end{eqnarray*}
Hence,
\begin{equation}\label{EQstar}
\lim_{\delta\to 0} \|u - v\|_{L^2}(t) =0,
\end{equation}
provided that (\ref{EQcclaim}) holds.

We observe that $w$ satisfies
\begin{equation}\label{rnie}
	w_t = \Delta^3w - \Delta^2( w(u^2 + uv + v^2) - w) + \delta(u^2)_x
	\,.
\end{equation}
Testing this equation with $(-\Delta)^{-2}w$ and integrating, we find
\begin{align*}
\frac12 \frac{d}{dt} ||(-\Delta)^{-1}w||_{L^2}^2
&=
	\IP{(-\Delta)^{-2}w}{\Delta^3w - \Delta^2( w(u^2 + uv + v^2) - w) + \delta(u^2)_x}
\\
&=
	  \IP{w}{\Delta w}
	+ \vn{w}_{L^2}^2
	- \IP{w}{w(u^2 + uv + v^2)}
	+ \delta \IP{(-\Delta)^{-2}w}{(u^2)_x}
\,.
\end{align*}
Note that $u^2 + uv + v^2 \ge \frac{u^2+v^2}{2} > 0$, so we set $\phi = \sqrt{u^2 + uv + v^2} > 0$. 
Then
\begin{equation}
\label{EQc1}
\frac12 \frac{d}{dt} ||(-\Delta)^{-1}w||_{L^2}^2
	+ \vn{w_x}_{L^2}^2
	+ \vn{w\phi}_{L^2}^2
=
	  \vn{w}_{L^2}^2
	+ \delta \IP{(-\Delta)^{-2}w}{(u^2)_x}
\,.
\end{equation}
The first term on the RHS needs to be dealt with.

We obtain another energy estimate by multiplying (\ref{rnie}) with $(-\Delta)^{-3}w$ and integrating:
\begin{align*}
\frac12 \frac{d}{dt} ||(-\Delta)^{-\frac32}w||_{L^2}^2
&=
	-\IP{(-\Delta)^{-3}w}{\Delta^3w - \Delta^2( w(u^2 + uv + v^2) - w) + \delta(u^2)_x}
\\
&=
	- \vn{w}_{L^2}^2
	+ \IP{(-\Delta)^{-1}w}{w(u^2 + uv + v^2 - 1)}
	- \delta \IP{(-\Delta)^{-3}w}{(u^2)_x}
\,.
\end{align*}
Adding this equality to \eqref{EQc1} cancels out the troublesome
$\vn{w}_{L^2}^2$ term. Integrating the result (recall that $w$ vanishes at the initial time) yields
\begin{align*}
\frac12 ||(-\Delta)^{-\frac32}w||_{L^2}^2(t)
&+ \frac12 ||(-\Delta)^{-1}w||_{L^2}^2(t)
	+ \int_0^t \big(\vn{w_x}_{L^2}^2(s) + \vn{w\phi}_{L^2}^2(s)\big)\,ds
\\
&=
	  \int_0^t \IP{(-\Delta)^{-1}w}{w(\phi^2 - 1)}(s)\,ds
	+ \delta N
\,,
\end{align*}
where we have set
\[
N(u,v) = \int_0^t \IP{[(-\Delta)^{-2} - (-\Delta)^{-3}]w}{(u^2)_x}(s)\,ds\,.
\]	
Now, using the Poincar\'e inequality and dropping some good positive terms we find
\begin{equation}
\label{EQc2}
\frac12 ||(-\Delta)^{-1}w||_{L^2}^2(t)
	+ \min\{1, L^{-2}\}\int_0^t\big( \vn{w}_{L^2}^2(s)
	+  \vn{w\phi}_{L^2}^2(s)\big)\,ds
\le
	  \int_0^s \vn{(-\Delta)^{-1}w}_{L^2}(s)\vn{w(\phi^2 - 1)}_{L^2}(s)\,ds
	+ \delta N
\,.
\end{equation}
Let us notice that
$$
\| w \|^2_{L^2} +  \| w \phi\|^2_{L^2} = \| w (\phi^2 + 1) ^{1/2}\|^2_{L^2} 
$$
and
$$
\vn{w(\phi^2 - 1)}_{L^2} \le \vn{w(\phi^2 + 1)^{1/2}}_{L^2} \|(\phi^2 + 1)^{1/2}\|_{L^\infty}. 
$$
Set $\gamma = \min\{1, L^{-2}\}^{1/2}$  and
$$
a(t) = \gamma \vn{w {(\phi^2+1)^{1/2}}}_{L^2}(t),
\qquad
b(t) = \frac12\vn{(-\Delta)^{-1}w}_{L^2}(t), \qquad
c(t) = {\frac 2\gamma \|(\phi^2 +1)^{1/2}\|_{L^\infty}(t)}.
$$
Lemma 3.1 in \cite{kory12} states that an inequality of the form
\[
	\int_0^t a^2(s)\,ds + b^2(t)
	\le \int_0^t a(s)b(s)c(s)\,ds + \delta N(t)
\]
implies
\[
	b^2(t) \le 2\delta N(t)\exp\Big( \int_0^t c^2(s)\,ds \Big)
\]
and
\[
	\int_0^t a^2(s)\,ds \le 2\delta N(t)\int_0^t  c^2(s)\exp\bigg( \int_0^t c^2(s)\,ds \bigg)\,ds + 2\delta N(t)
\,.
\]
Note that for any fixed $t>0$ we have uniform boundedness of $N$.
Moreover, the estimate 
(\ref{eq-unib}) and Remark \ref{rem-unib}
imply that
\begin{eqnarray*}
 \int_0^t c^2(s)\,ds &\le &
 \frac{4}{\gamma^2} \int_0^t (2( \| u\|^2_{L^\infty}(s) + \| v\|^2_{L^\infty} (s)) + 1)\,ds \\
&\le&\frac 4{\gamma^2}  \int_0^t (2 L^2( \| u_x\|^2_{L^2}(s) + \| v_x\|^2_{L^2} (s)) +1)\, ds\\
&\le& \frac{16L^2}{\gamma^2} t[(4L^2+1) C_2 +1],
\end{eqnarray*}
where $C_2$ depends on the radius
of  $\pHa2$-ball containing the compact set $K$.
Thus, the estimate \eqref{EQcclaim} follows.

{\it Step 2.} We notice that for any $t\ge 0$
$$
\lim_{\delta\to 0} \| u - v\|_{H^1} (t)=0
$$
uniformly with respect to $u_0\in K$. Indeed, using the notation from the previous step we see the integration by parts yields
$$
\|\nabla w\|_{L^2}^2 = \int_{\bT} \nabla w \nabla w\, dx = 
-\int_{\bT} w\Delta w\, dx.
$$
Now, our claim follows from (\ref{EQstar}) in the first step and the uniform boundedness of $u$, $v$ in $Z_2$.

{\it Step 3.} Treating $u$ and $v$ as given, we see that (\ref{rnie}) is a linear eq. for $w$ with zero initial condition. Then, using the notation of Step 1, the analog of the constant variation formula, cf. (\ref{eq-convar}), yields,
\begin{align*}
\| w \|_{Z_2}(t) &\le C \int_0^t \| (-\Delta)^{2\frac12} 
e^{\Delta^3(t-s)} \nabla ( w [\phi^2 -1])(s)\|_{L^2}\, ds +
\delta C \int_0^t \| -\Delta e^{\Delta^3(t-s)}  u(s) u_x(s)\|_{L^2}\,ds\PR{=:I_1+\delta I_2}.
\end{align*}
By Step 2, the integrand \PR{of $I_1$} goes to zero uniformly in $u_0\in K$ for each $s\in [0,t]$. The integrand is bounded by an integrable function, so we may pass to the limit with $\delta\to 0$. \PR{We also notice that $I_2$ is bounded uniformly in $\delta$, so $\delta I_2$ goes to zero.}
\end{proof}

\subsubsection{Part (d) of Theorem \ref{thmRob}}\label{subsec-d}

The first part of (d) on convergence of the set of equilibria as $\delta\rightarrow0$ is Corollary \ref{rmkHD}. Here, we have to use the topology of $Z_2$. We notice that the
range of the values of \GW{the} parameter $\delta$ \GW{has been so far restricted} to \GW{lie in} an interval $(0, \delta_1)$, see Theorem \ref{t-sm-dep} and Lemma \ref{co-dim0}.
It remains to prove the second part of (d), which is a kind of Liouville theorem for eternal solutions.
To this end, we show: 
\begin{theorem}\label{TMgap} {Let us suppose that $\delta_1$ is fixed as above.}
There exist\GW{s} $\eta>0$ 
such that for all $\delta\in(0,\delta_1)$ if $\xi_\delta:\bR \to \cA_\delta$ is an eternal solution to (\ref{eqn:hcch}), then
\begin{equation}\label{hyped}
 \sup_{t\in \bR} \dist_{\pHa2}(\xi_\de(t), \cE^{\delta,j})\le \eta \Rightarrow
 \xi_\de(t)\in \cE^{\delta,j}\quad \hbox{for all } t\in \bR.
\end{equation}
\end{theorem}
We will separately treat the cases where $\cE^{\delta,j}$ (for $j>0$) and $\cE^{\delta,0}=\{0\}$.
First, we deal with $j>0$ and consider a uniform tubular neighbourhood {of $\cE^{\delta,j}$} in $\pHa2$ with radius $r$,
\[
(\mathcal{E}^{\delta,j})_r = \{ v\in \pHa2: \dist_{\pHa2}(v, \mathcal{E}^{\delta,j})<r\}
\,.
\]
The constant $\eta$ in Theorem \ref{TMgap} above (which is universal) gives a `nonexistence gap' for eternal solutions: there can not
exist any non-trivial eternal solutions within $(\SE^{\delta,j})_\eta$.

When working with the equilibrium set and its tubular neighbourhoods, one
fundamental fact is that for $\eta$ small enough, the nearest point projection
is well-defined.

\begin{lemma}\label{l5.11}
There exists an $\eta^*$ such that for all $\eta < \eta^*$ the nearest point projection $\pi:(\SE^{\delta,j})_\eta\rightarrow\SE^{\delta,j}$ is well-defined.
\end{lemma}

\begin{proof}
In this proof we omit $\delta, j$ superscripts on $\SE$.
The proof follows by applying a generalised inverse function theorem (see e.g. \cite{NirenbergBook}) to a \GW{suitable} map.
First, let us recall the notion of normal bundle:
\[
	N\SE = \Big\{
		(u,v) \in Z_2 \times Z_2\, :\, u\in\SE\,,\ v\in N_u\SE
		\Big\}
		\subset T\pHa2
\]
where $N_u\SE$ is the orthogonal complement of $T_u\SE$ ($u\in\SE$) in $T_u\pHa2 \simeq \pHa2$. 

Set $h:N\SE\rightarrow \pHa2$ to be
\[
	h(u,v) = u+v
	\,.
\]
The curve $\SE$ is a $C^1$-submanifold of $\pHa2$; to see this, we can take as an
embedding (the restriction to the principal period of) the shift map centred at
$u_0$: $\SE = \overline\SE([0,T))$ where $\overline\SE$ is the map $s \mapsto
\tau_su_0$.

Then $dh_{(u,0)}$ is non-singular, because it maps $T_{(u,0)}(\pHa2\times\{0\})$
bijectively onto $T_u\SE$ (note that $T_{(u,0)}(\pHa2\times\{0\}) \subset
T_{(u,0)}N\SE$) and maps $T_{(u,0)}(\{u\}\times N_u\SE)$ bijectively onto
$N_u\SE\subset T_u\pHa2$.
Note that $h$ maps $\SE\times\{0\}$ diffeomorphically onto $\SE$.

The inverse function theorem then implies that $h$ maps an open neighbourhood $U$ of $\SE\times\{0\}$ in $N\SE$ diffeomorphically onto a neighbourhood $V$ of $\SE$ in $\pHa2$.
Compactness of $\SE$ implies that there exists an $\eta^*$ such that $(\SE)_{\eta^*} \subset U$.

Define $\pi:(\SE)_{\eta^*}\rightarrow\SE$ by
\[
	\pi(y) = (i\circ h^{-1})(y)
\]
where $i(u,v) = u$.
The map $\pi$ is a submersion because $i$ is a submersion and $h^{-1}$ is a diffeomorphism.

To see that $\pi(y)$ is the closest point in $\SE$ to $y$, suppose that there is another point $z\in\SE$ such that $z$ is closest to $y$.
The sphere centred at $y$ with radius $|y-z|$ is tangential to $\SE$ at $z$. Thus $y-z\perp \SE$ at $z$, so $y-z \in N_z\SE$.
This means
\[
	y = z + (y-z) = h(z, y-z)
\]
or $\pi(y) = z$.
\end{proof}
Due to Lemma \ref{l5.11} we may write $\pi(\xi(t)) = \tau_{s(t)}  u^{\delta, j} $, where  $s\in[0,L_{\text{prin}})$ is unique and $L_{\text{prin}}$ is the principal period. 

Let us fix $\xi_0 = \xi_\delta(t_0)$ and set $u_0 =\pi(\xi_0)$. By Hale-Raugel theory, see Section \ref{s-sa}, there exists a connected neighbourhood $V$ containing $u_0$
such that the decomposition
\begin{equation}\label{EQ1}
	\xi_\delta(t) = (x(t),y_1(t),y_2(t))
\end{equation}
holds, where each of the coordinates satisfy a uniformly parabolic evolution equation
with a specific form (see (\ref{co-rs}).
The coordinates are chosen with the decomposition
\[
	Z_\PR{2} = X\oplus Y_1\oplus Y_2
\]
and corresponding projection operators $P_0$, $P_1$ and $P_2$ in mind.
Note that $X$ and $Y_2$ are finite-dimensional spaces.
Let us call $V$ a \emph{Hale-Raugel neighbourhood} of $u_0$.

We work in a Hale-Raugel neighborhood $U$ centred at $u_0$.
This means that the origin of the coordinate system $(x, y_1, y_2)$ corresponds
to $u_0$. We may assume (by shrinking $U$ and $\eta$ if needed) that
\begin{equation}\label{eter0}
U= B_{\pHa{2}}(u_0, \rho)\cap (\mathcal{E})_\eta  =: U_{\rho,\eta}(u_0)
\,,
\end{equation}
where $\eta<\rho$.

Before we restrict our attention on the dynamics in $U_{\rho,\eta}(u_0)$ we will make a more general comment. 
At any $v^{s}:= \tau_{s}  u^{\delta, j} $ the dynamics in the corresponding Hale-Raugel neighborhood is in the form (\ref{co-r}). In order to stress that (\ref{co-r}) depends upon a specific choice of $s$ we will write the superscript $s$.
That is:
\begin{eqnarray}\label{co-rs}
 \dot x^s &=& \cC^s x^s +f^s(x^s,y^s_1,y^s_2),\nonumber\\
 \dot y^s_1 &=& \cB_1^s y^s_1 + g^s_1(x^s,y^s_1,y^s_2),\\
 \dot y^s_2 &=& \cB_2^s y^s_2 + g^s_2(x^s,y^s_1,y^s_2),\nonumber
\end{eqnarray}
where $\cC^s = P_0^s \cL^\delta_{v^s} = 0$, $\cB_1^s = P_1^s  \cL^\delta_{v^s}$, $\cB_2^s = P_2^s  \cL^\delta_{v^s}$.

We recall that 
\begin{equation}\label{co-sp}
\Re \sigma (\cB_1^s) \le \lambda_0<0,\qquad \Re \sigma (\cB_2^s) \ge \mu_0>0. 
\end{equation}
Lemma \ref{lem33}  implies 
\begin{equation*}
 \| g^s_1(0, y_1^s, y_2^s) \|_{L^2} + \| g^s_2(0, y_1^s, y_2^s) \|_{L^2}\le 
R(\| y_1^s\|_{\pHa5}, \| y_2^s\|_{\pHa5})\qquad 
R(\| y_1^s\|_{\pHa5}, \| y_2^s\|_{\pHa5})/(\| y_1^s\|_{\pHa5} + \| y_2^s\|_{\pHa5}) \to 0 .
\end{equation*}
However, we will need estimates in different norms.

Now, we study the \emph{transition time} $\Delta t$: this is the amount of time
taken for our eternal solution $\xi$ to traverse a Hale-Raugel neighbourhood
of this form.
It turns out that we can lengthen this time arbitrarily
by taking $\eta$ and $\rho$ small enough. In the forthcoming analysis we shall suppress the superscript $s$ that appears in the system (\ref{co-rs}).

\begin{proposition}
\label{PNwait} 
Let us suppose that $u_0$ is a steady state belonging to $\cE^{\delta,j}$, $j>0$. Then,
for each $M>0$ there exists a $\overline\rho<1$ with the following property.
For all $\rho \le \overline{\rho} $ and $U_{\rho,\eta}(u_0)$ defined in (\ref{eter0}) 
for $\eta = \rho^\beta<\eta^*$, where $\eta^*$ is given in Lemma \ref{l5.11}, and any $\beta>1$,
if $\xi_\delta$ is the eternal solution, then
\[
	\Delta t \ge M\,,
\]
where $\Delta t$ denotes the time taken for the solution $\xi_\delta$ to traverse $U_{\rho,\eta}(u_0)$.
\end{proposition}
\begin{proof}
As before, let us omit the $\delta, j$ superscripts.
We also omit the subscript $\delta$ on $\xi$.
Let us recall that $\pi \xi \in \SE^{\delta,j}$ 
and so $-A\pi\xi + F_\delta(\pi \xi) = 0$.
Therefore
\[
\xi_t = -A\xi + F_\delta (\xi) = - A\xi + F_\delta (\xi) + 
A\pi\xi -F_\delta (\pi\xi))
= -A(\xi - \pi\xi) + \int_0^1 DF_\delta(u^s)\,ds \, (\xi - \pi\xi)
\]
where $u^s = (1-s)\xi + s\pi \xi$. If we take into account the form of $F_\delta$, see  (\ref{rn:nl}) and Lemma \ref{lem222}, then we see
via the H\"older inequality the estimate
\begin{equation}
\label{EQwaitint}
	\vn{\xi_t}_{\pHa{-4}} \le \vn{\xi-\pi \xi}_{\pHa2}
	+ C \vn{\xi-\pi \xi}_{L^2}(\|\xi\|^2_{\pHa2} + \|\pi\xi\|^2_{\pHa2} + 1)
\,.
\end{equation}
Since $\pi\xi$ is the closest point projection, then our assumptions imply
\[
	\vn{\xi_t}_{\pHa{-4}} \le c\eta,
\]
where $c =c(\cE^{\delta,j},\eta).$
Let us now choose $\eta = \rho^\beta$ for any $\beta>1$.
This means
\[
	\vn{{\xi}_t}_{\pHa{-4}} \le c\rho^\beta 
\,.
\]
As the diameter of the Hale-Raugel neighbourhood is $c\rho$, because this is the size of $U_{\rho,\eta}(u_0)\cap \cE$ in any norm, we find that
\[
\Delta t
	\ge \frac{c\rho}{\vn{{\xi}_t}_{\pHa{-4}}} \ge \overline{c} \rho^{1-\beta}
\,.
\]
We can see that as $\rho\searrow0$, $\Delta t\rightarrow\infty$.
\end{proof}
Now, we turn our attention to the study of the norm of the eternal solution.
By the construction of the Hale-Raugel coordinate system the $x$-axis is
tangent to $\cE^{\delta,j}$ at $u_0$. Moreover, $\cE^{\delta,j}$ may be represented as the graph of a smooth function of $x$ in $U_{\rho,\eta}(u_0)$,
$$
U_{\rho,\eta}(u_0) \cap \cE^{\delta,j} = \{ (x, h(x)): \| x \|_{{Z_2}} < \rho \}
\,,
$$
where $h: X \to Y_1 \oplus Y_2$ is in $C^\infty$. In $U_{\rho,\eta}(u_0)$ we have dynamics given by \eqref{co-rs}, where we dropped the superscript $s$, which on $\SE^{\delta,j}$ takes the form
\begin{equation}\label{eter3}
 \begin{array}{l}
   0= f (x, h_1(x), h_2(x)),\\ 
   0= \cB_1 h_1(x) + g_1(x, h_1(x), h_2(x)),\\
  0= \cB_2 h_2(x) + g_2(x, h_1(x), h_2(x))
\,.
 \end{array}
\end{equation}
For an eternal solution $\xi_\delta(t) = (x(t), y_1(t), y_2(t))$,
we may use (\ref{eter3}) to recast (\ref{co-rs}) into a more useful
form.

For this purpose we  set $\tilde y = y - h(x)$. Here, we use the shorthand $y
= (y_1, y_2)$ and $h(x) = (h_1(x),h_2(x))$.
Hence,
\begin{equation}\label{eter4}
 \tilde y' = \cB \tilde y + g (x, \tilde y + h(x)) - g (x,  h(x)) =: 
 \cB \tilde y + G(x, \tilde y ) \tilde y \,,
\end{equation}
where $\cB \tilde y = (\cB_1 \tilde y_1, \cB_2 \tilde y_2)$ and
\begin{equation}\label{eter5}
 G(x, \tilde y )\tilde y = \int_0^1 \frac{d}{ds} g(x, s \tilde y + h(x))\,ds
\,.
\end{equation}
We have to estimate $v:= G(x, \tilde y ) \tilde y$ in a suitable norm. In order to do this, we have expose the structure of $g_1$, $g_2$ in (\ref{eter3}). We have to keep in mind that as $v$ belongs to a projected space, the image of $P_1+P_2$, we may not use directly the Sobolev norm but one which commutes with projection $P_1 + P_2$. In other words we have to work with the fractional powers of $\cL^\delta - \lambda=:\Lambda$, where $\cL^\delta$ is defined in (\ref{def-cL}), i.e. it is the linearisation at $u_0$ and $\lambda>0$ is big enough to guarantee that  $\Lambda$ has trivial kernel. Since we eventually want to estimate $\tilde y$ in the $\pHa{2}$-norm, we shall see that it is sufficient to bound 
$\| \Lambda^{-1/2} v\|_{L^2}$.

\begin{lemma}\label{lete1}
If $G$ is defined over $U_{\rho, \eta}(u_0)$ by (\ref{eter5}) and $\eta$, $\rho$ $\Lambda$ are as above, then
{
$$
\|\Lambda^{-1/2}  G(x, \tilde y ) \tilde y \|_{L^2} \le M \|\tilde y \|_{\pHa1},
$$}
where $M\le C \rho$.
\end{lemma}Note that the $x$-dependence in the estimate of Lemma \ref{lete1}  is now
multiplicative (if we estimate $g_1$ and $g_2$ in \eqref{eter3} the dependence
on $x$ is additive, although the growth in $y$ is quadratic).

\begin{proof}
Let us write $x+ h$ for $(x, h)$, where $x\in X,$ $h \in Y_1\oplus Y_2$. We recall that
$$
g_i(x, h) = P_i (F_\delta(u_0+x+h) - F_\delta(u_0)- DF_\delta(u_0)(x+h))
= P_i \int_0^1 \int_0^1 s D^2 F_\delta(u_0+s\sigma (x+h))(h,h)\,dsd\sigma, \qquad i=1,2.
$$
Since the nonlinearity $F_\delta$ is cubic,  $D^2F_\delta$ is linear.  Thus, $G(x,\tilde y)\tilde y$ takes the following form,
\begin{align*}
G(x,\tilde y)\tilde y= 
(P_1+P_2) & \int_0^1 \int_0^1 \int_0^1 s r 
\left(D^3F_\delta (u_0+s\sigma (x+h(x)+r\tilde y) )(x+h(x)+r\tilde y, x+h(x)+r\tilde y, \tilde y ) \right. \\ &+ \left.
2 D^2 F_\delta(u_0+s\sigma (x+h(x)+r\tilde y) )(x+h(x)+\tilde y, \tilde y\right) \,dsd\sigma dr.
\end{align*}
By  definition the projections $P_1$ and $P_2$ commute with $\Lambda$. Thus, if we take into account the form of $G(x,\tilde y)\tilde y$, then we reach
$$
\| \Lambda^{-1/2} G(x, \tilde y ) \tilde y \|_{L^2}
 \le C\|y \|_{\pHa{1}}  ( \|x \|_{\pHa1} + \|\tilde y \|_{\pHa1}),
$$
where $C$ is a universal constant.
  
Since $\|x \|_{\pHa2}\le \rho$ and $\|\tilde y \|_{\pHa2} \le \eta$, we reach our claim. Here, however, we use the  observation below, see Lemma \ref{l0}. 
Note that we have already taken $\eta = \rho^\beta$, where $\beta>1$, and $\rho <\overline{\rho}$ is smaller than one (cf. the hypothesis needed for the waiting time estimate, Proposition \ref{PNwait}).
\end{proof}

A technical difficulty arises from the fact that each connected component of our set of equilibria is not flat. As a direct consequence the Hale-Raugel decomposition at any point $u_0\in\SE^{\delta,j}$ is not expected to be orthogonal.
This brings into question the amount of control we have on the orthogonal distance. The following lemma deals with this issue.

\begin{lemma}\label{l0}
Let $u_0\in \SE^{\delta,j}$. There exists a Hale-Raugel neighbourhood $V$ \eqref{eter0} of $u_0$ corresponding to the $X\oplus Y_1\oplus Y_2$ decomposition of $Z_2=\pHa2$ such that
\begin{equation*}
 \dist_{{\pHa2}} (\xi_\delta, \cE^{\delta,j}) \le \|\tilde y \|_{{\pHa2}} \le c\dist_{{\pHa2}} (\xi_\delta, \cE^{\delta,j})
\,,
\end{equation*}
where $(x,y_1,y_2)$ are the Hale-Raugel coordinates in $V$, $\xi_\delta = (x,y_1,y_2)$, $\tilde y = y - h(x)$, $h(x)$ is such that $(x,h(x))\in\SE^{\delta,j}$ in $V$ and $c>0$ is a universal constant. 
\end{lemma}
\begin{proof} 
As before, let us omit the $\delta, j$ superscripts on $\cE$ and $\delta$ subscript on $\xi$. 
First, note that the endpoints of the vector $\tilde y$ are $\xi$ and $(x,h(x))\in \cE$. Hence, by the definition of $\dist_{\pHa2} (\xi, \cE)$ we have
$$
\|\tilde y \|_{{\pHa2}} \ge \dist_{{\pHa2}} (\xi, \cE).
$$
It remains to show that $\dist_{\pHa2}(\xi, \cE)$ is bounded from below by $c\|\tilde y\|_{\pHa2}$.

The space $X$ is spanned by $e$, the normalised (in $Z_2$) derivative of $u_0$. We avoid on purpose naming this variable in order to avoid a clash of notation. Since the decomposition $X\oplus Y$ need not be orthogonal, we introduce the normal space $N =\{ e\}^\perp$. We now consider the following dichotomy:\\ (1) $N = Y$, then for $y\in Y$ we have $\| y\| =\dist((0,y)+u_0, \cE)$; and\\ (2) $N\neq Y.$

Let us first consider case (2).
Then we have
\begin{equation}
\text{dim }( ( N\cap Y
)^\perp) = 2
\,.
\label{eq2}
\end{equation}
Indeed, since $N \ne 
{Y}$, there is a non-zero vector $w$ perpendicular to $Y$ such that the set $\{w, e\}$ is linearly independent.

We notice that any  $z\in N+ u_0$ may represented as follows with the help of the Hale-Raugel coordinates, 
\[
xe + y  = z - u_0 , 
\]
where we use the  above notation and keep in mind that  $x$ is a scalar. 
We may take the inner product in  $Z_2$ (denoted by $\cdot$) of both sides with $e$. Taking into account that $z-u_0\in 
{N}$ yields
\begin{equation}
\label{eq3}
x + y\cdot e = 0\,.
\end{equation}
Let us use $d$ to denote the distance from $z$ to $\SE$. Then
\[
    d^2 = ||z-u_0||^2
        = ||xe + y||^2
        = x^2 + ||y||^2 + 2x e\cdot y\,.
\]
Using now \eqref{eq3} we have
\begin{equation}\label{eq5-20p}
    d^2 = ||y||^2 - x^2
\,.
\end{equation}
We claim that there exists an $\alpha_0>0$ such that
\begin{equation}
\label{eqOnN}
\frac{d}{||y||} > \alpha_0\qquad\text{on $N$}\,.
\end{equation}
This estimate is invariant under scaling, so we may restrict our attention to the unit sphere in $N$, i.e. we have to prove that
$$
\frac 1{\|y\|} >\alpha_0>0
$$
for $y\in N$ belonging to the unit sphere. As the sphere is a bounded set this claim clearly holds.

Therefore we have \eqref{eqOnN} on $N$. It remains to extend this inequality off $N$, and in doing so it will be essential that we replace $y$ by $\tilde y = y - h(x)$. 
For points $(x, y)\in Y\cap N$, we have $h(x) =0$, hence $\tilde y = y$.
Thus, by the definitions of the object involved in the formula below, we have
\begin{equation}
\label{eq_onN}
\alpha_0 < \frac{d}{||\tilde y||} = 
\frac{\dist_{Z_2}((x,y) + u_0, \cE)}{\| y - h(x)\|},\qquad (x,y) \in Y\cap N.
\end{equation}
Note that the projections $P_0$, $P_1$, $P_2$, the map
$(x,y)\mapsto \tilde y$, the norm and the distance are all continuous.
Using this continuity, we may extend inequality \eqref{eq_onN} to the estimate
\begin{equation}\label{eq5-23}
\frac{\dist_{Z_2}((x,y) + u_0, \cE)}{\| y - h(x)\|} > \frac{\alpha_0}{2}.
\end{equation}
in a neighbourhood of $u_0$.
This finishes the proof in case (2). 

In the first case where $N=Y$, the estimate \eqref{eqOnN} again holds with {any} constant $\alpha_0$ {less than} 1. The argument proceeds from this point exactly as in case (2) above.
\end{proof}
Now, we complete the estimate that will give us the control we need on the waiting time $\Delta t$.
The constant variation formula applied to (\ref{eter4}) and each component of $\tilde y$ separately yields
\begin{equation}\label{eter6}
 \tilde y_1(t) = e^{\cB_1 (t-t_0)} \tilde y_1(t_0) + 
 \int_{t_0}^t e^{\cB_1 (s-t_0)} (G(x, \tilde y ) \tilde y)_1(s)\,ds\,.
\end{equation}
We recall that $\Re \sigma(\cB_1)< -\lambda_0 <0$, then (\ref{eter6}) and Lemma \ref{lete1} yield,
\begin{equation}\label{eter7}
 \| \tilde  y_1(t) \|_{\pHa2} \le M e^{-\lambda_0(t-t_0)} \| \tilde  y_1(t_0) \|_{\pHa2}
 + M \int_{t_0}^t \frac{e^{-\lambda_0(s-t_0)} }{ (s - t_0)^{1/2}}
 \| \tilde y \|_{\pHa2}(\| \tilde y \|_{\pHa1} + \| x \|_{\pHa1})\,ds\,. 
 \end{equation}
Due to the choice of $\eta$ made in Proposition \ref{PNwait}, we have
$$
\| \tilde y \|_{\pHa2} + \| x \|_{\pHa2} \le \rho + \eta \le 2 \rho
\,.
$$
We further restrict here $\rho$ by requiring that
\begin{equation}\label{eter9}
 2 M \rho \int_{0}^\infty \frac{e^{-\lambda_0 s} }{ s ^{1/2}}\,ds < \frac 18.
\end{equation}
Since (due to Lemma \ref{l0}) $\tilde y_1$ is comparable with $\eta$, and we can assume that $\| \tilde y_1\|_{\pHa2} \le \rho^{\beta}$. Thus, (\ref{eter9}) yields
\begin{equation}\label{eter10}
 \| \tilde y (t)\|_{\pHa2} \le  M e^{-\lambda_0(t-t_0)} \rho^{\beta}
 + \frac 18 \rho^{\beta}
\end{equation}
We notice that for large $t$ inequality (\ref{eter10}) gives us an improvement of the estimate
o $\| \tilde y (t)\|_{\pHa2}$. If we pick $t_c$ such that
$$
M e^{-\lambda_0 t_c} = \frac 18\,,
$$
that is
\begin{equation}\label{eter11}
 t_c = \frac{\ln(8M)}{\ln \lambda_0}\,;
\end{equation}
then we obtain that
\begin{equation}\label{eter12}
\| \tilde y (t)\|_{\pHa2} \le \frac 14 \rho^{\beta}\qquad\hbox{for }t> t_0 +t_c\,.
\end{equation}
However, we have to know that the interval 
$(t_0 + t_c, t_1)$ is not empty. We may compare $t_c$ and $\Delta t$ using the following estimate:
$$
\Delta t \ge C \rho^{1-\beta} >  t_c = \frac{\ln(8M)}{\ln \lambda_0}\,.
$$
This is another restriction on the size of $\rho$.

We will now estimate $\| \tilde y_2 (t)\|_{\pHa2}$.
Using again the constant variation formula, we find
\begin{equation}\label{eter6a}
 \tilde y_2(t) = e^{\cB_2 (t-t_0)} \tilde y_2(t_1) + 
 \int_{t_1}^t e^{\cB_1 (s-t_1)} (G(x, \tilde y ) \tilde y)_2(s)\,ds\,.
\end{equation}
This time, however, time will run backward. We may switch the direction of time here because $Y_2$ is finite dimensional. Moreover, $\Re \sigma(\cB_2)\ge \mu_0>0$, hence
estimating the norm of both sides gives us
\begin{equation}\label{eter7a}
 \| \tilde  y_2(t) \|_{\pHa2} \le M e^{\mu_0(t-t_1)} \| \tilde  y_1(t_1) \|_{\pHa2}
 + M \int_{t}^{t_1} \frac{e^{-\mu_0(t_0-s)} }{ (t_1 - s)^{1/2}}
 \| \tilde y \|_{\pHa2}(\| \tilde y \|_{\pHa1} + \| x \|_{\pHa1})\,ds. 
 \end{equation}
Arguing like before we arrive at
\begin{equation}\label{eter10a}
 \| \tilde y_2 (t)\|_{\pHa2} \le  M e^{-\mu_0(t_1-t)} 
 + \frac 18 \rho^{\beta}
\end{equation}
for $t\in (t_0, t_1)$. In our estimates we may assume that $\mu_0 = - \lambda_0$, then for $t< t_1 - t_c$ inequality (\ref{eter10a}) gives us an improvement of the estimate of  $\| \tilde y_2 (t)\|_{\pHa2}$,

Now, for $t\in (t_0 +t_c, t_1 -t_c)$ we obtain a simultaneous improvement of $\| \tilde y_i (t)\|_{\pHa2}$, $i=1,2$,
\begin{equation}\label{eter13}
 \| \tilde y(t)\|_{\pHa2} \le \frac 12 \rho^{\beta}\qquad\hbox{for }t\in( t_0 +t_c, t_1- t_c).
\end{equation}
We notice that due to our choice of $\Delta t $ and $t_c$ the interval $( t_0 +t_c, t_1- t_c)$ is not empty.

Finally, we prove the main result {of this section}.

\begin{proof}[Proof of Theorem \ref{TMgap}]
We shall apply our work above with a certain choice of the endpoints $t_0$ and $t_1$.

Let us suppose that there exists $t_*$ such that 
\begin{equation}\label{eter-f}
\dist_{\pHa2}(\xi_\delta(t_*),\cE^{\delta,j})> \frac
34 \rho^\beta. 
\end{equation}
If there is no such $t_*$, 
then we may further restrict $\rho$ and $\eta$
through Proposition \ref{PNwait}
and perform this argument again.

We position our Hale-Raugel neighborhood $U_{\rho, \rho^\beta}(u_0)$ in a such a way
that 
$\xi_\delta(t_*)\in U_{\rho, \rho^\beta}(u_0)$ and 
also $t_*\in (t_0 +t_c, t_1- t_c)$.
However, the improvements on the distance $\dist(\xi_\delta, \cE^{\delta,j})$ we showed in (\ref{eter13}) contradict (\ref{eter-f}). 

It remains to deal with the left out case of  $j=0$ corresponding to the isolated equilibrium consisting of the single point $\{u_0\}$ where $u_0(x)=0$. In this situation, we may conduct the same proof just after the choice of a sufficiently small neighborhood of $u_0$ yielded by Proposition \ref{PNwait}.
Our claim follows.
\end{proof}

\subsection{Proof of the main result}
\label{S52}
When we set $\delta =0$, then we have already noticed that (\ref{eqn:hcch}) is a gradient flow of the functional $\cF$ (defined in (\ref{eqncF})) with respect to the $(\dot H^2_{per})^*$ topology. In particular, there are no homoclinic orbits and eternal solution connect critical points of $\cF$, which are sets $\cE^i$, $i=1,\ldots, k.$

In subsection \ref{s5.1} we verified assumptions (a)--(d) of Theorem 
\ref{thmRob} in relation to invariant sets of \GW{the} semigroup $S_\delta$. Thus, there exist\GW{s} $\delta_2>0$  such that for all $\delta\in (0, \delta_2)$ equation (\ref{eqn:hcch}) is a gradient flow with respect to the family of invariant sets $\cS^\delta$, $\delta\in (0,\GW{\delta_2})$. In particular, $\cS^\delta$ does not contain any homoclinic orbits. In other words, $S_\delta$ is a gradient semigroup. As a result we may apply Theorem \ref{TH-R} to $S_\delta$. We have already checked assumptions  (H.1)--(H.4). We \GW{need} to know that if $z_0$ is a steady state, then $\omega(z_0)$ consists only of the steady states of (\ref{eqn:hcch}). But Theorem \ref{thmRob} implies that $\omega(z_0)\subset \cS^\delta$. Since $\cS^\delta$ does not contain homoclinic structures, we deduce that $\omega(z_0)$ contains only equilibria, i.e. 
\begin{equation}\label{limit}
    \lim_{t\to\infty} u(t) = \phi\in Z_{\PR{2}}\qquad \hbox{in the }Z_{\PR{2}}\hbox{ topology}.
\end{equation}
Now, we pick any $k\in\bN$, $k>0$. We see that
$$
\| (-\Delta)^{k/2}(u(t) - \phi)\|_{L^2}^2 = \int_\bT (-\Delta)^{k/2}(u(t) - \phi) \cdot (-\Delta)^{k/2}(u(t) - \phi) \,dx =
- \int_\bT (-\Delta)^{k}(u(t) - \phi) \cdot(u(t) - \phi)\,dx ,
$$
where we performed one integration by parts. Furthermore,
$$
\| (-\Delta)^{k/2}(u(t) - \phi)\|_{L^2}^2\le 
\| (-\Delta)^{k}(u(t) - \phi)\|_{L^2} \cdot \|u(t) - \phi\|_{L^2}.
$$
If we combine this with (\ref{limit}) and 
(\ref{eq-unib})
then we deduce that $u(t)$ converges to $\phi$ in the $\pHa{k}$ topology.
Our claim follows. 

\subsection*{Acknowledgements}
The authors thank the sponsoring institutions for their support. PR is grateful to the University of Wollongong for creating a very rewarding working environment, especially through the VISA and Distinguished Visitor Programs. A big part of work on this project was done at UoW.
The work of PR was also in part supported by the National Science Centre, Poland, through the grant  2017/26/M/ST1/00700. GW was in part supported by the National Science Centre, Poland, trough the grant 2015/19/P/ST1/02618 during his visit to the Banach Center, Warsaw, where a part of work was done.
GW was additionally partially supported by ARC grant DP150100375 and the Institute for Mathematics and its Applications.
The authors are grateful for this sustained support from a number of sources.

Special thanks go to our colleagues. We are grateful to Grzegorz Lukaszewicz for bringing \cite{robinson} 
to our attention. We thank Henryk Zol{\c a}dek for stimulating discussion on the monotonicity of the period function. 
Galina Pilipuk was so kind to perform the Mathematica computations for us.


\begin{thebibliography}{33}
\setlength{\itemsep}{-2mm} 
\bibitem{bates-fife}
P.W.Bates, P.C.Fife, Spectral comparison principles for the Cahn-Hilliard and phase-field equations, and time scales for coarsening, {\it Phys. D}, {\bf 43} (1990), no. 2-3, 335--348.
\bibitem{bates}
P.W.Bates, K.Lu, A Hartman-Grobman theorem for the Cahn-Hilliard and phase-field equations,
{\it J. Dynam. Differential Equations,} {\bf  6} (1994), 101--145. 
\bibitem{robinson}
A.N.Carvalho, J.A. Langa,
J.C.Robinson,
{\it Attractors for infinite-dimensional non-autonomous dynamical systems},
Springer, New York, 2013.
\bibitem{chl}
L.Cherfils , A.Miranville and Sh.Peng, 
Higher-order anisotropic models in phase separation
{\it Adv. Nonlinear Anal.} {\bf 8} (2019), no. 1, 278--302
\bibitem{hale-raugel} 
J.K.Hale, G.Raugel, Convergence in gradient-like systems with applications to PDE. {\it Z. Angew. Math. Phys.}, {\bf 43} (1992), no. 1, 63--124.
\bibitem{henry} D. Henry,
{\it Geometric theory of semilinear parabolic equations}, Lecture Notes in Mathematics, 840. Springer, 
Berlin-New York, 
1981.
\bibitem{KPZ}
M.Kardar, P.Garisi,  Y.C.Zhang, Dynamic scaling of growing interfaces, {\it Phys Rev Lett.}  {\bf 56} (9) (1986), 889--892.
\bibitem{kato}
T.Kato, Perturbation theory for linear operators. Reprint of the 1980 edition. Classics in Mathematics. Springer-Verlag, Berlin, 1995.
\bibitem{korzec-evans}
M.D. Korzec, P.L. Evans, A.M\"unch, and B. Wagner,
Stationary Solutions of Driven Fourth- and Sixth-Order Cahn--Hilliard-Type Equations,
{\it SIAM J. Appl. Math.,} {\bf  69}, (2008), 348--374.
\bibitem{konary15}
M. D. Korzec, P. Nayar and P. Rybka,
Global attractors of sixth order PDEs describing the faceting of growing surfaces, 
{\it J. Dynam. Differential Equations},
{\bf 28}, (2016), 49-67. \\ M.D. Korzec, P.Nayar, P.Rybka, Correction to: Global Attractors of Sixth Order PDEs Describing the Faceting of Growing Surfaces
{\it  J. Dynam. Differential Equations} (2019).
\bibitem{kory12}
M. Korzec, P. Rybka, On a higher order convective convective Cahn-Hilliard type equation, {\it SIAM J. Appl. Math.} {\bf  72}, (2012), 1343-1360. 
\bibitem{konary12}
M. Korzec, P. Nayar, P. Rybka, Global weak solutions to a sixth order Cahn-Hilliard type equation, {\it SIAM J. Math. Analysis,} {\bf  44}, (2012), 3369-3387.
\bibitem{TW-chhc}
Z.Li, Ch. Liu,  On the nonlinear instability of traveling waves for a sixth-order parabolic equation, {\it Abstr. Appl. Anal.} (2012), Art. ID 739156, 17 pp.
\bibitem{chhc-periodic}
Ch.Liu, A.Liu and H.Tang,
Time-periodic solutions for a driven sixth-order Cahn-Hilliard type equation, 
{\it Boundary Value Problems} (2013) 2013:73, 17 pp.
\bibitem{13} S.Lojasiewicz,   Une propri\'et\'e topologique des
sous-ensembles analytiques re\'els. Colloque Internationaux du C.N.R.S. \#
117, Les equations aux deriv\'ees partielles (1963), 87-89.
\bibitem{14} S.Lojasiewicz,  Sur la geometrie semi- et sous-analytique.
{\it Ann. Inst. Fourier (Grenoble)}  {\bf 43,} (1993), 1575--1595. 
\bibitem{miranville} A.Miranville, Alain Sixth-order Cahn-Hilliard systems with dynamic boundary conditions, {\it Math. Methods Appl. Sci.}  {\bf  38} (2015), no. 6, 1127--1145.
\bibitem{NirenbergBook} L. Nirenberg,
	{\it Topics in nonlinear functional analysis}, {\bf 6}, American Mathematical Society, 1974.
\bibitem{paza}  I.Pawlow, W.M.Zaj\c aczkowski, The global solvability of a sixth order Cahn-Hilliard type equation via the B\"acklund transformation, {\it Commun. Pure Appl. Anal.}  {\bf 13} (2014), no. 2, 859--880.
 \bibitem{peng}  Sh.Peng, H.Zhu, Well-posedness for modified higher-order anisotropic Cahn-Hilliard equations,  {\it Asymptot. Anal.}  {\bf 111} (2019), no. 3-4, 201--215.
\bibitem{hory} P.Rybka, K.-H.Hoffmann Convergence of solutions to Cahn-Hilliard equation, {\it Commun. PDE.} {\bf 24} (1999), 1055--1077.

\bibitem{savina} T.~V. Savina, A.~A. Golovin, S.~H. Davis, A.~A. Nepomnyashchy and P.~W. Voorhees, Faceting of a growing crystal surface by surface diffusion,
{\it Phys. Rev. E}, {\bf 67},
(2003), 021606
\bibitem{BsimonVol4} B.Simon, {\it Operator Theory, A Comprehensive Course in Analysis, Part 4}, American Mathematical Society, Providence, RI, 2015.
\bibitem{wangliu}
Zh.Wang, Ch. Liu,
Some properties of solutions for the sixth-order Cahn-Hilliard-type equation
{\it Abstr. Appl. Anal.}, (2012), Art. ID 414590, 24 pp.
\end{thebibliography}
\end{document}